\documentclass[11pt]{article}

\usepackage[mathscr]{eucal}
\usepackage[margin=1in]{geometry}
\usepackage{amsrefs}
\usepackage{amsmath}
\usepackage{amsthm}
\usepackage{amsfonts}
\usepackage{amssymb}
\usepackage{hyperref}
\usepackage{enumerate} 
\numberwithin{equation}{section}

\newcommand\blfootnote[1]{%
 \begingroup
 \renewcommand\thefootnote{}\footnote{#1}%
 \addtocounter{footnote}{-1}%
 \endgroup
}

\newtheorem{theorem}{Theorem}[section]
\newtheorem{lemma}[theorem]{Lemma}
\newtheorem{proposition}[theorem]{Proposition}
\newtheorem{corollary}[theorem]{Corollary}

\theoremstyle{definition}
\newtheorem{example}[theorem]{Example}
\newtheorem{definition}[theorem]{Definition}
\newtheorem{assumption}[theorem]{Assumption}

\newtheorem{remark}[theorem]{Remark}

\newcommand{\Emb}{{\mathbb{E}}}
\newcommand{\Fmb}{{\mathbb{F}}}
\newcommand{\Gmb}{{\mathbb{G}}}
\newcommand{\Hmb}{{\mathbb{H}}}
\newcommand{\Imb}{{\mathbb{I}}}

\newcommand{\Pmb}{{\mathbb{P}}}

\newcommand{\Rmb}{{\mathbb{R}}}

\newcommand{\Zmb}{{\mathbb{Z}}}

\newcommand{\Bmc}{{\mathcal{B}}}

\newcommand{\Dmc}{{\mathcal{D}}}
\newcommand{\Emc}{{\mathcal{E}}}

\newcommand{\Gmc}{{\mathcal{G}}}
\newcommand{\Hmc}{{\mathcal{H}}}

\newcommand{\Jmc}{{\mathcal{J}}}
\newcommand{\Kmc}{{\mathcal{K}}}

\newcommand{\Nmc}{{\mathcal{N}}}

\newcommand{\Pmc}{{\mathcal{P}}}

\newcommand{\Smc}{{\mathcal{S}}}

\newcommand{\Zmc}{{\mathcal{Z}}}

\newcommand{\Nms}{{\mathscr{N}}}

\newcommand{\Pms}{{\mathscr{P}}}

\def\R{{\mathbb R}}

\def\N{{\mathbb N}}
\def\PP{{\mathbb P}}

\def\FF{{\mathbb F}}

\def\P{{\mathcal P}}

\def\X{{\mathcal X}}

\def\Z{{\mathbb Z}}

\def\F{{\mathcal F}}

\def\pspace{(\Omega, {\mathcal F}, \Pmb)} 
\def\fpspace{(\Omega, {\mathcal F}, \FF, \Pmb)}

\newcommand{\filt}{\Fmb}
\newcommand{\filtm}{\F}
\newcommand{\nfilt}{\Gmb}
\newcommand{\nfiltm}{\Gmc}
\newcommand{\hfilt}{\Hmb}
\newcommand{\hfiltm}{\Hmc}

\newcommand{\Xref}{\wh{X}}
\newcommand{\mref}{\wh{\mu}}

\newcommand{\rateref}{\wh{\rate}}
\newcommand{\locrate}{\alt{\rate}}
\newcommand{\wlocrate}{\ov{\rate}}
\newcommand{\irate}{\rate}
\newcommand{\irateref}{\rateref}
\newcommand{\Lambdaref}{\wh{\Lambda}}
\newcommand{\Psiref}{\wh{\Psi}}
\newcommand{\lambdaref}{\wh{\lambda}}

\newcommand{\jmp}[2]{{\rm Disc}_{#2}\left(#1\right)}
\newcommand{\thx}[2]{t_{#1}(#2)}
\newcommand{\vx}[2]{v_{#1}(#2)}
\newcommand{\jx}[2]{j_{#1}(#2)}
\newcommand{\thi}[1]{t_{#1}}
\newcommand{\vi}[1]{v_{#1}}
\newcommand{\ji}[1]{j_{#1}}

\newcommand{\leb}{{\rm Leb}}
\newcommand{\Sm}[1]{\#_{#1}}	

\newcommand{\jmps}{\Jmc}

\newcommand{\cad}{\Dmc}

\newcommand{\Pol}{\Zmc}

\newcommand{\neigh}[1]{\Nmc_{#1}}	
\newcommand{\kneigh}[2]{\Nmc^{#1}_{#2}}	
\newcommand{\gneigh}[2]{\Nmc_{#1}(#2)}	
\newcommand{\kgneigh}[3]{\Nmc^{#1}_{#2}(#3)}	
\newcommand{\cl}[1]{\te{cl}_{#1}}			
\newcommand{\gcl}[2]{\te{cl}_{#1}(#2)}			
\renewcommand{\root}{\o}
\newcommand{\subg}[1]{[#1]}	
\renewcommand{\sp}[1]{[#1]}		

\newcommand{\poiss}{\mathbf{N}}
\newcommand{\pp}{p}				
\newcommand{\rpp}{P}				

\newcommand{\rate}{r}
\newcommand{\rateset}{\mathbf{r}}
\newcommand{\x}{\xi}					
\newcommand{\vtx}{\x}
\newcommand{\vms}{\mathbf{\kappa}}		
\newcommand{\dvms}{\vartheta}		
\newcommand{\mksp}{\Kmc}

\newcommand{\mspacenew}{\widetilde{\Omega}, \widetilde{{\mathcal F}}} 
\newcommand{\filtnew}{\widetilde{\mathbb{F}}}
\newcommand{\pfiltnew}{\widetilde{{\mathcal F}}}
\newcommand{\Pnew}{\widetilde{\mathbb{P}}}
\newcommand{\RNnew}{\rho}

\newcommand{\skipLine}{\vspace{12pt}}
\newcommand{\ov}{\overline}
\newcommand{\te}{\text}
\newcommand{\indp}{\perp\!\!\!\perp}
\newcommand{\indic}[1]{\Imb_{\left\{#1\right\}}}
\newcommand{\ex}[1]{\Emb\left[#1\right]}
\newcommand{\exnd}[1]{\Emb[#1]}		
\newcommand{\exmu}[2]{\Emb^{#1}\left[#2\right]}	
\newcommand{\deq}{\overset{\text{(d)}}{=}}		
\newcommand{\defeq}{:=}
\newcommand{\mdeg}{\alpha}
\newcommand{\borel}{\Bmc}
\newcommand{\wh}[1]{\widehat{#1}}
\newcommand{\ind}{\hspace{24pt}}
\newcommand{\alt}[1]{\widetilde{#1}}
\newcommand{\law}{\te{Law}}

\allowdisplaybreaks
\title{Interacting Jump Processes Preserve Semi-Global Markov Random Fields on Path Space}
\author{Ankan Ganguly and Kavita Ramanan, \\
 Division of Applied Mathematics, Brown University}
\date{}

\begin{document}

\maketitle
\begin{abstract}
Consider a system of interacting particles indexed by the nodes of a graph whose vertices are equipped with marks representing parameters of the model such as the environment or initial data. Each particle takes values in a countable state space and evolves according to a (possibly non-Markovian) continuous-time pure jump process whose jump
intensities depend only on its own state (or history) and marks as well as the
states (or histories) and marks of
particles and edges in its neighborhood in the graph. Under mild conditions on the jump intensities, it is
shown that the trajectories of the interacting particle system exhibit a certain local or semi-global Markov random field property
whenever the initial condition satisfies the same
property.
Our results complement recent works that establish the preservation of a
local second-order Markov random field property for interacting diffusions.
Our proof methodology in the context of jump processes is different, 
and works directly on infinite graphs, thereby bypassing any limiting arguments. Our results apply to models arising in diverse fields including statistical physics, neuroscience, epidemiology and opinion dynamics, and have 
direct applications to the study of 
marginal distributions of interacting particle systems on Cayley trees. 
\end{abstract}
\blfootnote{\emph{2000 Mathematics Subject Classification.} Primary: 60K35, 60J74, 60J80; Secondary: 60K25;}
\blfootnote{\emph{Key words and phrases.} Interacting particle systems; jump processes; Poisson random measures; Markov Random Field; Point processes}
\blfootnote{This material is based upon work supported in part by the
Vannevar Bush Faculty Fellowship ONR-N0014-21-1-2887 and 
the 
United States Army Research Office under grant number W911NF2010133. }

\section{Introduction}
\label{intro}

A pure jump interacting particle system (IPS) describes a collection of randomly evolving particles indexed
by the nodes of an underlying graph, where the dynamics of each particle in the collection is described by a pure jump process on a discrete state space, with jump rates depending not only on its own state but
also on the states of neighboring particles in the graph. 
Commonly studied IPS include the voter model, contact process and Glauber dynamics for various
statistical physics spin models like the Ising and Potts models \cite{Lig85}, as well as many other models
arising in engineering and operations research (see \cite{GanRam-Hydro22} for a list of references).
Several works over the last two decades have shown that such IPS do not preserve Gibbsianness \cite{Vanetal02,KisKul20,Kul19,Vanetal12, Vanetal10, JahKul17,FerHolMar13}. In fact, as demonstrated in \cite{Vanetal02,Vanetal12,KisKul20,Vanetal10,FerHolMar13}, even if the initial states of particles form a Markov random field (MRF) with respect to the underlying interaction graph,
the collection of states of all particles at some future time may fail to form a Markov random field (MRF)
of any order (with respect to the same graph). In fact, this can occur even if the IPS is ergodic with a stationary distribution that is an MRF (see, e.g. \cite{Vanetal02}). As is well known, given a locally finite graph $G = (V,E)$ and a Polish space $\Pol$, the $\Pol^V$-valued random element $Z$ is said to form a (local) MRF if for any disjoint partition $A,B,S$ of $V$ such that $S$ is equal to $\gneigh{A}{G}$, the neighborhood of the finite set $A$ in $G$,
\begin{equation}
Z_A\indp Z_B|Z_S.
\label{intro:MRFprop}
\end{equation}
Furthermore, for any $\mdeg \in \N$, $Z$ is said to form an $\mdeg$-MRF if instead $S$ is equal to $\kgneigh{\mdeg}{A}{G}$, the $\mdeg$-neighborhood of the finite set $A$ in $G$ (which is the set of vertices in $A^c$ that lie at a distance of at most $\mdeg$ from $A$) above. On infinite graphs, one can also consider global MRFs: an MRF or $\mdeg$-MRF is said to be \emph{global} if \eqref{intro:MRFprop} holds even when $A$ is infinite. Instead, we introduce the intermediate notion of a semi-global MRF (SGMRF), which will turn out to be more relevant for our purposes. An SGMRF or $\mdeg$-SGMRF is defined analogously to the MRF (respy. $\mdeg$-MRF) property except that \eqref{intro:MRFprop} must hold even for infinite $A$ whose $\mdeg$-neighborhood $S$ is finite. In this article we show (under general conditions on the jump intensities and the interaction graph) that
if the initial states
form an $\mdeg$-MRF (respy. $\mdeg$-SGMRF) with $\mdeg \geq 2$, then the {\em trajectories} also form an $\mdeg$-MRF (respy. $\mdeg$-SGMRF).
In particular, this shows that
there is preservation of the $\mdeg$-MRF and $\mdeg$-SGMRF properties at the level of trajectories, even if
not at the level of states. 
In addition, we also show this in general fails to hold if $\alpha = 1$ (see Example \ref{res:1-MRF}).

\ind The definition of an $\mdeg$-SGMRF is a natural extension of the definition of a ``Markov chain on a tree,'' as stated in \cite[Section 2]{Zac83} and \cite[Chapter 12]{Geo11}, to higher-order Markov chains and general graphs (see Appendix \ref{MRFalt} for further discussion). An important motivation for establishing the second-order
SGMRF property is that it can be used to obtain autonomous descriptions of
marginal dynamics for IPS on trees as unique solutions to certain associated local equations \cite[Chapter 6]{Gan22}.
As shown in \cite[Theorem 4.3 and Corollary 4.7]{GanRam-Hydro22},
the local equations describe the limit of both the neighborhood empirical measure
as well as the marginal dynamics
at the root of IPS on sequences of uniformly rooted random regular graphs whose sizes
grow to infinity, much in the spirit of mean-field limits for IPS on complete graphs \cite{Oel84}.
The MRF property by itself is insufficient for such a characterization (as also observed in \cite{LacRamWuLE21} in the context of diffusions).

\ind
Our results in fact apply to a far more general class of IPS characterized as solutions to Poisson-driven stochastic differential equations (SDEs) that may be non-Markovian or heterogeneous. Non-Markovian dynamics are crucial to model a variety of applications in
neuroscience, epidemiology and engineering,
and heterogeneities arise naturally in many settings, including load balancing models \cite{AghRam19,Fonetal21,TodTru19}. We capture heterogeneities in the dynamics by equipping the interaction graph with
(possibly random) marks on the vertices, which
specify the initial states and/or
initial histories of the particles, as well as heterogeneities in the dynamics,
random environments and asymmetries in the local interactions with respect to the neighboring particles.
The jump rates of
each particle are allowed to depend on the histories of neighboring particles as well as the marks
of vertices in the neighborhood (a precise model description is given in Section \ref{mod}).
When the rates satisfy some mild regularity conditions (stated in Assumptions \ref{mod:regular} and
\ref{mod:WPbd}), our main result (Theorem \ref{res:MRF}) shows that if the random marks form an MRF or SGMRF of order
$\alpha \geq 2$, then the trajectories of the IPS also exhibit the same MRF property.

\ind To the best of our knowledge, this article is the first exploration of MRF properties of trajectories of IPS described by jump processes. However, there exist results of a similar flavor in the context of diffusions (see \cite{Deu87,DerRoe17,CatRoeZes96,LacRamWuMRF21} and references therein). Specifically, Theorem 2.7 of \cite{LacRamWuMRF21} establishes conditions under which trajectories of interacting diffusions preserve the second-order {\em local}
MRF property. Our results generalize those of \cite{LacRamWuMRF21} in the jump process context by considering higher-order MRFs, as well as SGMRFs
rather than just MRFs, and 
weakening assumptions on the initial data, allowing for more general initial data than the initial state of the process.
Specifically, unlike in \cite{LacRamWuMRF21}, we do not require that the initial conditions be absolutely continuous with respect to any product measure. This is of particular interest in the study of marginals of stationary Markov processes, as well as non-Markov processes, for which the initial data includes the history of the process before time zero. 
These ``infinite histories'' are typically highly singular so that even on finite graphs, the initial data will typically fail to be absolutely continuous with respect to any product measure.

\ind
Despite some similarity in the results, it is worth emphasizing that our proof technique differs from that used for diffusions in \cite{LacRamWuMRF21}. In the latter work, the trajectories of interacting diffusions are first shown to preserve the 2-MRF property on finite graphs, and then a limiting argument
is used to extend to infinite graphs. This argument exploits the fact that interacting diffusions on infinite graphs arise as local weak limits of interacting diffusions on finite graphs \cite[Theorem 3.7]{LacRamWuLWC22}, and shows that the
$2$-MRF property is preserved along suitably constructed convergent sequences. This approach requires one to impose certain assumptions about the continuity of the dynamics with respect to the initial condition (to ensure the aforementioned local weak convergence).
In contrast, we directly prove our main result for IPS on infinite graphs without invoking of limiting arguments. A brief outline of our approach, which allows us to handle both the MRF and SGMRF properties in a unified manner, is as follows. First, given an IPS
we construct an associated sequence of reference processes on infinite graphs, whose (initial data and) trajectories satisfy a certain conditional independence property that is akin to an MRF property. We then show that the law of the IPS is absolutely continuous with respect to that of each of the reference processes and use the form of the Radon-Nikodym derivative to deduce the MRF or SGMRF property
of the IPS from the conditional independence properties of the reference processes. An intermediate step in this process that may be of independent interest is an infinite-dimensional Girsanov theorem for IPS on possibly non-locally finite graphs even when the defining SDE may have multiple weak solutions (see Proposition \ref{rn:girs}). The proof of this proposition proceeds by first establishing a duality between the IPS and a point process and then applying extensions of standard results for non-explosive point processes to the explosive marked point process setting to deduce the resulting Radon-Nikodym derivative.

\ind In Section \ref{nota} we establish basic definitions and notations that will be in use throughout the article. In Section \ref{mod-main}, we introduce the model, state our main results and provide certain counterexamples that suggest our results cannot, in general, be reasonably strengthened. In Section \ref{MRF}, we prove the main result taking for granted the absolute continuity of the IPS with respect to the reference processes and the form of the associated Radon-Nikodym derivatives. The latter are derived in Section \ref{sec:rn:rnd}. In Appendix \ref{ap-condind}, we prove a technical lemma that is used to describe the conditional structure of the reference processes. In Appendix \ref{MRFalt} we provide an alternate characterization of the SGMRF property and derive associated properties. In Appendix \ref{Ap:dual} we show that the point process dual is a well-defined property. Lastly, in Appendix \ref{WP}, we apply the results of \cite{GanRam-Hydro22} to prove that the the sequence of reference IPS is well defined under the conditions we impose upon it.

\section{Preliminaries and Notation}
\label{nota}

For any real numbers $a,b\in \R$, we write $a\wedge b \defeq \min\{a,b\}$ and $a \vee b \defeq \max\{a,b\}$.

\skipLine

\textbf{Graph Notation: } Given a set $A$, let $|A|$ denote its cardinality. Let $G \defeq (V,E)$ represent a graph, with countable vertex set $V$ and edge set $E$. Graphs are always assumed to be simple (i.e., they do not have self-loops or multi-edges) and undirected. For $u, v \in V,$ a path between $u$ and $v$ in $G$ is defined to be a sequence of vertices $u = v_0, v_1,\dots,v_{n-1},v_n = v$ for some $n \in \N_0$ such that for all $i \in\{1,\dots,n\}$, $\{v_{i-1},v_i\} \in E$ and $v_i \neq v_j$ whenever $i \neq j$ except possibly when $(i,j) = (0,n)$, in which case the path is said to be a cycle. The length of the path is the number of edges in the path. Let $d_G(u,v)$ denote the usual graph distance, which is the length of the shortest path between $u$ and $v$ in $G$. If there are no paths between $u$ and $v$, then $d_G(u,v) = \infty$. Note that for $v \in G$, the sequence $\{v_0 = v\}$ is a path of length 0 so that $d_G(v,v) = 0$.

\ind For any subset $U \subseteq V$, let $\neigh{U} \defeq \gneigh{U}{G}\defeq\{v \in V\setminus U: \{u,v\} \in E \te{ for some }u \in U\}$ denote the neighborhood of $U$ in $G$ and let $\cl{U}\defeq \gcl{U}{G} \defeq U\cup\neigh{U}$
denote its closure. Also, for $\mdeg \in \N$, the set $\kneigh{\mdeg}{U} \defeq \kgneigh{\mdeg}{U}{G} \defeq \{v \in V\setminus U: \min_{u \in U}d_G(u,v) \leq \mdeg\}$ denotes the $\mdeg$-neighborhood of $U$.
If $U = \{v\}$ is a singleton and the graph is clear from the context, then we write $\neigh{v}\defeq \gneigh{v}{G} \defeq \gneigh{\{v\}}{G}$, $\kneigh{\mdeg}{v}\defeq \kgneigh{\mdeg}{v}{G} \defeq \kgneigh{\mdeg}{\{v\}}{G}$ and $\cl{v}\defeq \gcl{v}{G} \defeq \gcl{\{v\}}{G}$. The degree of a vertex $v$ is equal to $|\neigh{v}|$, the graph $G$ is said to be locally finite if each of its vertices has finite degree and
the graph $G$ is said to be of bounded degree if $\sup_{v \in V} |\neigh{v}| < \infty$. Unless otherwise specified, all graphs are assumed to be locally finite. On occasion, we may slightly abuse notation by writing $v \in G$ to mean $v \in V$. For $U \subseteq V$, let
$G\subg{U}$ denote the induced subgraph of $G$ on $U$, that is, $G\subg{U} \defeq (U,E\subg{U})$ with $E \subg{U} \defeq \{ \{u,v\} \in E: u,v \in U\}.$

\skipLine

\textbf{Path Space Notation:} Given any countable index set $U$ and Polish space $\Pol$, let $\Pol^U = \{(z_v)_{v \in U}: z_v \in \Pol\te{ for all } v \in U\}$
denote the corresponding configuration space, equipped with the product topology. 
For any $z \in \Pol^V$, $z_U \in \Pol^U$ denotes the restriction of $z$ to $\Pol^U$, that is, $z_U = (z_v)_{v\in U}$. We consider IPS with a countable state space $\X$, which we identify with a subset of the integers $\Z$ equipped with the discrete topology.
Given $U\subseteq V$ and a (closed, half-open or open) interval $I \subseteq [0,\infty)$, let $\cad(I,\X^U)$ denote the space of c\`adl\`ag functions from $I$ to $\X^U$. Given 
$0 < t < \infty$, and $I = [0,t]$ or $I = [0,t)$, 
for conciseness denote $\cad(I,\X^U)$ by $\cad^U_t$ or $\cad^U_{t-}$, respectively. 
Also, set $\cad^U \defeq \cad([0,\infty),\X^U)$
and omit the superscript $U$ from the notation when $|U| = 1$. If $x \in \cad^U$ and $v\in U$, then $x_{v}(t)$ denotes the value of the $v$th component of $x$ at time $t \geq 0$. For any $t \geq 0$, the restrictions of $x$ to $[0,t]$ and $[0,t)$ are respectively denoted by $x[t]\in \cad^U_t$ and $x[t)\in \cad^U_{t-}$. Also, set $\Delta x(t) \defeq x(t) - x(t-)$. For $0\leq s < \infty$, an interval $I \subseteq [0,\infty)$, finite $U\subseteq V$ and $x \in \cad(I,\X^U)$, define the set of jump times as follows:
\begin{equation}
\label{nota:disc}
\jmp{x}{s} \defeq \{s' \in I\cap (0,s]: x(s') \neq x(s'-)\}, 
\end{equation}
and for $x \in \cad^U$, set $\jmp{x}{} := \cup_{s \in (0,\infty)} \jmp{x}{s}$. 
For $t \in [0,\infty]$, $\cad^U_t$ is equipped with the product J1 topology, under which $\cad^U_t$ is a Polish space \cite[Theorems 12.2 and 16.3]{Bil99}. Note that for $\{x_n\} \subseteq \cad^U$, $x_n \to x$ if $x_n[t] \to x[t]$ for every $t \notin \jmp{x}{}$. Next, for any fixed $t \in \R_+$ and any strictly increasing locally Lipschitz function $\psi: [0,t)\to [0,\infty)$ with a locally Lipschitz inverse (e.g., $\psi(s) \defeq \frac{1}{t-s}-\frac{1}{t}$), the bijection $\cad^U_{t-}\ni x \to x\circ\psi^{-1}\in \cad^U$ induces the J1 topology on $\cad^U_{t-}$.
This can be used to show that $\cad^U_{t-}$ is also Polish under the J1 topology.

\skipLine

\textbf{Measure Notation:} For any Polish space $\Pol$, let $\borel(\Pol)$ be the Borel $\sigma$-algebra on $\Pol$, and let $\Pmc(\Pol)$ be the space of probability measures on $(\Pol,\borel(\Pol))$ equipped with the topology of weak convergence. Given $U \subseteq V$ and $\eta \in \Pmc(\Pol^V)$, let $\eta[U]$ be the marginal distribution of $\eta$ restricted to $\Pol^U$. Given any $\eta\in \Pmc(\Pol)$ and a $\Pol$-valued random element $Z$, we say $Z \sim \eta$ if the distribution of $Z$ is given by $\eta$. We write $Y\deq Z$ if $Y$ and $Z$ have the same distribution, and write $Y_1\indp Y_2$ (respy. $Y_1\indp Y_2|Y_3$) if $Y_1$ and $Y_2$ are independent (respy. conditionally independent given $Y_3$). 
If $\eta \in \Pmc\left(\cad(\Rmb_+,\Pol)\right)$ for some Polish space $\Pol$, then $\eta_t\in \Pmc\left(\cad([0,t],\Pol)\right)$
and $\eta_{t-}\in \Pmc\left(\cad([0,t),\Pol)\right)$ denote the restrictions of $\eta$ to the respective Borel $\sigma$-algebras $\borel(\cad([0,t],\Pol))$ and $\borel(\cad([0,t),\Pol))$.

\ind A filtration is said to satisfy the usual conditions if it is complete and right-continuous. Unless otherwise stated, all filtrations
are assumed to be augmented so as to satisfy the usual conditions. Filtrations will typically be represented by the letters $\filt,\nfilt$ and $\hfilt$, indexed by $\R_+$ or $[0,T]$ and for each $t \in \R_+$, the corresponding $\sigma$-algebras will be denoted by $\filtm_t,\nfiltm_t,\hfiltm_t$ respectively. Given a filtration $\nfilt \defeq\{\nfiltm_t\}_{t\in\R_{+}}$, recall that 
a simple sufficient condition for a process $Z$ to be $\nfilt$-\emph{predictable} is that $t\mapsto Z_t$ is almost surely left-continuous and $\nfilt$-adapted. Given two filtrations $\filt$ and $\nfilt$, as usual $\filt\vee\nfilt \defeq (\filtm_t\vee \nfiltm_t)_{t \in \R_+}$ denotes the smallest filtration containing both $\filt$ and $\nfilt$. Given a random element $\zeta$, we use $\hfiltm^\zeta$ to denote the completion of the $\sigma$-algebra generated by $\zeta$ (with respect to a probability measure that will be expressed explicitly if not clear from the context), and for any c\`adl\`ag stochastic process $Z$, we define $\hfilt^Z \defeq \{\hfiltm^Z_t\}_{t \in \R_+}$ to be the smallest filtration satisfying the usual conditions such that $Z$ is adapted to $\hfilt^Z$. For all processes $Z$ considered in this paper, $\hfilt^Z$ will be equal to the completion of the natural filtration of $Z$ so that for all $t \in \R_+$, $\hfiltm^Z_t = \hfiltm^{Z[t]}$; see the discussion in \cite[page 357]{DalVer08} for more details. 

\skipLine

\textbf{Poisson Point Processes:} Let $\Pol$ be a Polish space equipped with its Borel $\sigma$-algebra
and a metric $d_{\Pol}$ that induces the Polish topology. On intervals $I\subseteq \R_+$ and on countable spaces (which will be assumed to have an implicit embedding in $\N$), this will be the standard absolute difference metric, and on c\`adl\`ag spaces it will be the J1 metric. Finally, if $\Pol \defeq \prod_{i \in I}\Pol_i$ for some finite index set $I\subseteq\N$ with metrics $d_{\Pol_i}$, $i \in I$, then we set $d_{\Pol}(z,z') = \sum_{i \in I}d_{\Pol_i}(z_i,z'_i)$. Let $\Nms(\Pol)$ be the space of locally finite, nonnegative integer-valued measures, that is, for any $\pp \in \Nms(\Pol)$ and $A \in \borel(\Pol)$, $\pp(A) \in \N_0\cup\{\infty\}$ and $\pp(A) < \infty$ for every $A$ that is bounded with respect to $d_{\Pol}$.
We equip $\Nms(\Pol)$ with the weak-hash topology, which then makes it a Polish space 
\cite[page 2 and Proposition 9.1.IV(iii)]{DalVer08}, \cite{Mor18}.
Also, note that the map $\Nms(\Pol) \ni p\mapsto p(A)$ is Borel-measurable for any $A \in \borel(\Pol)$. 

\ind Let $\eta$ be any nonnegative, locally finite Borel measure on $\Pol$, that is, $\eta(B)<\infty$ for every bounded set $B \in\borel(\Pol)$. A \emph{Poisson point process} $\poiss$ on $\Pol$ with \emph{intensity measure} $\eta$ is a point process on
$\Pol$ such that for any disjoint $A,B \in \borel(\Pol)$, $\poiss(A)$ is Poisson distributed with expectation $\eta(A)$ and $\poiss(A)\indp \poiss(B)$.
A Poisson process $\poiss$ is said to be a $\nfilt$-Poisson point process on $\wh{\Pol}\defeq \R_+\times\Pol$ if for every $t \geq 0$, $A \in \borel([0,t]\times \Pol)$ and $B \in \borel((t,\infty)\times \Pol)$, $\poiss(A)$ is $\nfiltm_t$-measurable and $\poiss(B)\indp \nfiltm_t$. We let $\hfilt^{\poiss}$ denote the minimal filtration satisfying the usual conditions such that $\poiss$ is an $\hfilt^{\poiss}$-Poisson process.

\section{Model Description and Main Results}
\label{mod-main}

\subsection{Model Description and Assumptions}
\label{mod}

We consider IPS in which each particle takes
values in a countable state space ${\mathcal X}$ (viewed without loss of generality as a subset of $\Zmb$) and has state 
transitions that lie in some finite jump set 
$\jmps \subseteq \{i - j: i,j \in \X, i \neq j\}$. We restrict consideration to the case $|\jmps|<\infty$ because this setting leads to simpler and more transparent expressions and 
seems to cover most examples of interest, although it is straightforward to generalize our results to the case of a countable jump set $\jmps$. The data specifying the model consists of a deterministic (simple, locally finite, undirected) graph $G = (V,E)$ that encodes the interaction structure, the initial data $\vms \in \mksp^V$, where $\mksp$ is a Polish space, 
and a family of jump rate functions 
$\rate^v_j:\R_+\times (\mksp\times\cad)^{V} \to\R_+, j \in \jmps, v \in V$, that specifies the dynamics, where $\cad$ is the space of c\`{a}dl\`{a}g functions taking values in $\X$ (using the notation from Section \ref{nota}). The initial data $\vms$ can not only capture the initial state at time zero (for a Markov process) or history before time zero (for a non-Markovian process), it can also be used to encode other state parameters of the model such as random environments and heterogeneities in particle dynamics (see \cite[Section 4.4]{GanRam-Hydro22} for concrete examples). 
We assume interactions between particles are local (with respect to the graph $G$), predictable and with regular paths,
as encapsulated in the following assumption.
\begin{assumption}
\label{mod:regular}
The family of rate functions 
$\rateset \defeq \{\rate^{v}_j\}_{v \in V,j \in \jmps}$ consists of Borel measurable functions from $\R_+ \times (\mksp\times\cad)^{V}$ to $\R_+$ that satisfy the following three conditions:
\begin{enumerate}
\item (locality) 
for every $v \in V$ and $j \in \jmps$, there exists a function $\locrate^v_j: \R_+\times (\mksp\times\cad)^{\cl{v}}$ such that for every $(t,\dvms,x) \in \R_+\times(\mksp\times\cad)^V$,
\[ \rate^v_{j}(t,\dvms,x) = \locrate^v_j(t,\dvms_{\cl{v}},x_{\cl{v}})
\qquad \mbox{ for all } (t,\dvms,x) \in \R_+\times(\mksp\times\cad)^V; \] 
\item (predictability) 
for every 
$\dvms \in \mksp^{V}$, $t > 0$, $v \in V$ and $x, y \in \cad^{V}$, 
\[ x(s) = y(s):\quad s \in [0,t) \quad \Rightarrow \quad \rate^v_j (t, \dvms,x) = \rate^v_j (t, \dvms,y); 
\]
\item (regularity) for every $v \in V, j\in \jmps$ and $(\dvms,x) \in (\mksp\times\cad)^V$, the map $(0,\infty)\ni t \mapsto \rate^v_j(t,\dvms,x) \in \R_+$ is c\`agl\`ad.
\end{enumerate}
\end{assumption} 

\ind
In what follows, $\leb$ is Lebesgue measure on $\R_+$ and $\Sm{\jmps}$ is the counting measure on $\jmps$. 
\begin{definition}
\label{mod:drive}
The \emph{solution space} associated with the $\mksp^V$-valued (random) initial data $\vms$ 
is a complete filtered probability space $\fpspace$ with $\filt$ satisfying the usual conditions, which supports the initial data $\vms$ with $\filtm_0 \supseteq \sigma (\vms)$, and a collection of i.i.d. $\filt$-Poisson processes $\poiss \defeq \{\poiss_v\}_{v \in V}$ on the space $\R^2_+\times\jmps$ with intensity measure $\leb^2\otimes \Sm{\jmps}$, referred to as \emph{driving Poisson processes}. 
\end{definition}

\ind 
Given the solution space $\fpspace$, initial data $\vms$ and jump rate function family $\rateset \defeq \{\rate^v_j\}_{v \in V,j\in\jmps}$ that satisfy Assumption \ref{mod:regular}, we now describe the associated IPS $X$ as a solution to the following Poisson-driven SDE: 
\begin{equation}
\label{mod:infpart}
X_v(t) = X_v(0) + \int_{(0,t]\times \R_+\times \jmps} j\indic{r\leq \rate^{v}_{j}(s,\vms,X)}\,\poiss_{v}(ds,dr,dj),\quad v \in V,t \in [0,\infty). 
\end{equation}
Note that Assumption \ref{mod:regular} implies that $(s,\vms,X) \mapsto \rate^{v}_{j}(s,\vms,X)$ only 
depends on $X$ and $\vms$ only via $X_{\cl{v}}[s)$ and $\vms_{\cl{v}}$.
Since, as mentioned above, the initial data $\vms$ may contain more than the initial state, we will find it convenient to express the latter as a Borel measurable function $\vtx: \mksp \to \X$ of the initial data:
\begin{equation}
\label{ic}
X_v(0) := \vtx (\vms_v), \quad v \in V.
\end{equation}
We call $\vtx$ the \emph{initial condition map}, and we will refer to $(\vms,\vtx)$ as the initial data pair.

\begin{definition}
\label{mod:WP} 
A \emph{weak solution} to the SDE \eqref{mod:infpart}-\eqref{ic} for the initial data $\vms$ is a $\filt$-adapted c\`adl\`ag stochastic process $X$ defined on an associated solution space $\fpspace$ 
that satisfies \eqref{mod:infpart}-\eqref{ic} almost surely. The SDE \eqref{mod:infpart}-\eqref{ic} is said to be \emph{strongly well-posed} for the initial data $\vms$ if on any solution space $\fpspace$ associated with $\vms$, 
there exists a weak solution $X$ to \eqref{mod:infpart}-\eqref{ic} for the initial data $\vms$ and the SDE \eqref{mod:infpart}-\eqref{ic} is pathwise unique in the sense that 
given any other weak solution $Y$ to \eqref{mod:infpart}-\eqref{ic} on the same solution space (and hence, with the same driving Poisson processes and initial data) it follows 
that $X = Y$ almost surely. 
\end{definition}

\ind 
Our main result holds under the following mild condition that the jump rate functions at a vertex
satisfy a certain degree-dependent bound. 
\begin{assumption}
\label{mod:WPbd}
There exists a function $C: \N\times \R_+\to \R_+$ that is non-decreasing in each of its arguments and such that for any $v \in V,j \in \jmps,k\in \N$ and $t \in \R_+$, $\rate^v_j(t,\cdot,\cdot) \leq C(|\cl{v}|,t)$. 
\end{assumption}

\begin{remark} 
\label{rem:WP}
Most IPS of interest satisfy
Assumptions \ref{mod:regular} and \ref{mod:WPbd} (e.g., see \cite[Section 4.4]{GanRam-Hydro22}). 
As shown in Lemma \ref{lem-ass}, under these assumptions it follows from 
\cite[Theorem 4.2, Propositions 5.15 and 5.17]{GanRam-Hydro22} that 
strong well-posedness of the IPS holds on a very large class of graphs that includes bounded degree graphs
and almost sure realizations of Galton-Watson trees
whose offspring distributions have finite first moments. Thus, the main result, Theorem \ref{res:MRF}, in this article pertains to the case in which the SDE \eqref{mod:infpart}-\eqref{ic} is strongly well-posed.
However, the change of measure result in Proposition \ref{rn:girs} and the duality result in
Proposition \ref{duo:duality} are established under more general assumptions that do not require \eqref{mod:infpart}-\eqref{ic} to be well-posed.
\end{remark}

\begin{remark}
\label{mod:whystrong}
The description of a weak solution in Definition \ref{mod:WP} corresponds to a $(\filt,\poiss)$-weak solution in
\cite[Remark 3.5]{GanRam-Hydro22}. However, when we consider settings in which the SDE \eqref{mod:infpart}-\eqref{ic} is strongly well-posed, it follows from 
\cite[Lemma 3.10]{GanRam-Hydro22} that every such weak solution is necessarily also \emph{strong} in the sense that it is adapted to the filtration $\hfiltm^{\vms}\vee\hfilt^\poiss$. Given this equivalence, we omit the qualifier ``weak'' or ``strong'' for solutions to \eqref{mod:infpart}-\eqref{ic} when the SDE is strongly well-posed. 
\end{remark}

\ind
Unless otherwise specified, we use $\fpspace$ and $\poiss$ to denote the solution space and driving Poisson processes of \eqref{mod:infpart}-\eqref{ic}. 

\subsection{Main Result and Counterexamples}
\label{res:res}

Our main result shows that the trajectories of the IPS propagate certain MRF properties. To state these precisely, recall the definition of MRF, $\mdeg$-MRF, SGMRF and $\mdeg$-SGMRF from Section \ref{intro}. 
\begin{theorem}
\label{res:MRF}
Suppose the jump rate function family $\rateset \defeq \{\rate^v_j\}_{v \in V, j\in\jmps}$ satisfies Assumptions \ref{mod:regular} and \ref{mod:WPbd}, and $G$ is either a graph of bounded degree or
an a.s. realization
of a Galton-Watson tree whose offspring distribution has a finite first moment. Given an initial data pair $(\vms,\vtx)$, let $X = X[\infty)$ be a solution to \eqref{mod:infpart}-\eqref{ic}.
If $\vms$ forms an $\mdeg$-MRF (respy. $\mdeg$-SGMRF) with respect to $G$ for some integer $\mdeg \geq 2$, then for each $t \in (0,\infty]$, $(\vms,X[t))$ forms an $\mdeg$-MRF (respy. $\mdeg$-SGMRF) with respect to $G$.
\end{theorem}
Theorem \ref{res:MRF} is a direct consequence of a more general result, Proposition \ref{gen:MRF}, which
establishes 
this ``preservation of MRF'' property for IPS on a broader class of graphs (that satisfy the condition stated in 
Assumption \ref{mod:WPassu}), and Lemma \ref{lem-ass}, which shows that this class 
includes the graphs mentioned in Theorem \ref{res:MRF}. 
The proof of Proposition \ref{gen:MRF} is presented in Section \ref{MRF:rnd}. 

\ind
Theorem \ref{res:MRF}
is used in forthcoming work \cite{GanRam-LETDet22} to obtain an autonomous characterization of the marginal distribution on the root neighborhood of an IPS with homogeneous jump rate functions on
the $d$-regular tree. As elaborated in the next section,
the theorem is also relevant to the study of Gibbs-non Gibbs transitions of spin models. 

\ind We now describe two counterexamples that demonstrate that the results in
Theorem \ref{res:MRF} cannot in general be improved. 
Specifically, the first example shows that the 
analog of Theorem \ref{res:MRF} does not in general hold when $\mdeg = 1$. 

\begin{example}
\label{res:1-MRF}
\emph{There exists an IPS $X$ on a finite graph $G = (V,E)$ with jump rate functions $\rateset$ and initial data pair $(\vms,\vtx)$ such that the components of $\vms$ are mutually independent and for which $(\vms,X[t))$ does not form an MRF for any $t > 0$.}

\skipLine

Let $G = (V,E)$ be the path on three vertices with $V = \{1,2,3\}$ and $E = \{\{1,2\},\{2,3\}\}$. Let $\mksp = \X = \{0,1\}$, $\jmps = \{1\}$. Also, let $\vtx$ be the identity map on $\mksp$. Set $X_2(0) = \vms_2 := 0$ and let $(X_1(0),X_3(0)) = (\vms_1,\vms_3)$ be i.i.d. Bernoulli($1/2$) random variables. Note that then $\{\vms_v\}_{v=1,2,3}$ are mutually independent.

\ind Consider the following (Markovian) jump rate functions, $\rate^v_1: \R_+\times \cad^{V} \to \R_+$:
\[\rate^1_1 \equiv \rate^3_1 \equiv 0 \quad \te{ and } \quad \rate^2_1(t,x) = \indic{x_1(0) \neq x_3(0),x_2(t-) = 0}.\]
This model satisfies Assumptions \ref{mod:regular} 
and \ref{mod:WPbd}, and the maximum degree of any vertex in $G$ is $2$. Now, $X_1(0) \neq X_3(0)$ with probability $1/2$ and $X_2 \equiv 0$ on the event $\{X_1(0) = X_3(0)\}$. Note that the function $f: \cad_{t-}\to \R$ defined by $f(y) = y(0)$ is bounded and
measurable, and on the event $\{X_2(t-) = 1\}$, which implies $X_1(0) \neq X_3(0)$, it is easily verified that 
\[
\begin{array}{rll}
\ex{f(X_1[t))|X_2[t)} &= \PP\left(X_1(0) = 1|X_2[t)\right) &= \frac{1}{2}, \\
\ex{f(X_1[t))|X_{\{2,3\}}[t)} &= \PP\left(X_1(0) = 1|X_{\{2,3\}}[t)\right) &= 1-X_3(0). 
\end{array}
\]
Since $X_3(0)$ is random, $\law(X_1[t)|X_{\{2,3\}}[t)) \neq \law(X_1[t)|X_2[t))$. Hence, the trajectories $X[t)$ do not form an MRF for any $t > 0$.
\end{example}

\ind 
Next we observe that for some $t > 0$, the states $X(t) = \{X_v(t)\}_{v \in V}$ (as opposed to the trajectories) 
may fail to form an $\mdeg$-MRF for any $\mdeg \in \N$ even if the initial data $\vms$ forms an SGMRF.
Indeed, this follows from the substantial literature on the topic of dynamic transitions of IPS from Gibbs to non-Gibbs states as demonstrated in the example below.
\begin{example}
\label{res:state}
\emph{ Fix $d > 2$. Let $G$ be the infinite $d$-regular tree and let $\nu$ be the positive-boundary ferromagnetic Ising model for an inverse temperature-magnetic field pair $(\beta,h)$ at which the Ising model experiences a phase transition (as described in \cite[Section 12.2]{Geo11}).
If $X$ is the IPS corresponding to infinite-temperature Glauber dynamics (as described in \cite{Vanetal02}) with initial condition 
$X(0) \sim \nu$, then $X(0)$ forms a 1-SGMRF, but there exists an interval $[t_1,t_2)\subseteq \R_+$ such that for all $t \in [t_1,t_2)$, $\{X_v(t)\}_{v \in V}$ does not form an $\mdeg$-order MRF for any $\mdeg \geq 1$.}\\

The fact that $X(0)$ is an MRF follows from the fact that it is a Gibbs measure associated with a Markovian specification (see \cite[Section 12.2]{Geo11}). Furthermore, as an extremal countable state MRF \cite[Theorem 12.31]{Geo11}, it is also a Markov chain on the tree \cite[Corollary 2]{Zac83}. By Lemma \ref{res:Markov}, this implies that it is also an SGMRF and therefore an $\mdeg$-SGMRF for every $\mdeg \in \N$.
The assertion of the example then follows from \cite[Theorem 3.9 and Remark 3.10]{Vanetal12}. Since $\law(X(t))$ is non-Gibbs (i.e., non-quasilocal) in the stated interval, $X(t)$ fails to form an $\mdeg$-MRF for any $\mdeg \in \N$.
\end{example} 

\section{A Generalization of the Main Result}
\label{MRF}

\subsection{Statement of the More General Result} 
\label{MRF:gen}

For simplicity of formulation, in Theorem \ref{res:MRF} we only addressed certain classes of graphs $G$.
However, as shown in Proposition \ref{gen:MRF} below, 
the conclusion of Theorem \ref{res:MRF} in fact holds
for any graph-jump rate function pair that satisfies the following more general 
(but less transparent) assumption.
The assumption is expressed in terms of a certain family of reference processes which are modified versions of the original 
family of jump rate functions $\rateset\defeq \{\rate^v_j\}_{v \in V,j\in\jmps}$
defined as follows: 
for any $W \subset V$, let $\wh{\rateset}^W \defeq \{\rateref^{W,v}_j\}_{v \in V,j\in\jmps}$ be the family of jump rate functions defined by
\begin{equation}
\label{mod:refrate}
\rateref^{W,v}_j(\cdot,\cdot,\cdot) \defeq \begin{cases}
\rate^{v}_j(\cdot,\cdot,\cdot) &\te{ if } v \notin W,\\
1 &\te{ otherwise.}
\end{cases}
\end{equation}
The property that each reference IPS is a {\em strong} solution to the
associated SDE plays a crucial role in the proof of 
Proposition \ref{gen:MRF} (see Remark \ref{mod:whystrong} and
Proposition \ref{ref:CI}). This motivates the following assumption.
\begin{assumption}
\label{mod:WPassu}
The graph $G$, the initial data pair $(\vms,\vtx)$ and the family of jump rate functions $\rateset\defeq \{\rate^{v}_j\}_{v \in V,j \in\jmps}$ are
such that for every finite (and possibly empty) $W\subseteq V$,
the SDE \eqref{mod:infpart}-\eqref{ic} is strongly well-posed for the initial data
$\vms$ when $\rateset$ is replaced with the modified 
jump rate function family $\wh{\rateset}^W\defeq \{\rateref^{W,v}_j\}_{v \in V,j \in\jmps}$ defined in \eqref{mod:refrate}. 
\end{assumption}

\ind 
The next lemma shows that this assumption holds under the conditions of
Theorem \ref{res:MRF}. 
\begin{lemma}
\label{lem-ass}
Suppose $\rateset \defeq\{\rate^{v}_j\}_{v \in V,j \in\jmps}$ satisfies 
Assumptions \ref{mod:regular} and \ref{mod:WPbd} and 
$G$ is either a graph with finite maximal degree or an
a.s. realization of a Galton-Watson tree whose offspring distribution
has a finite first moment. 
Then $G$, $(\vms,\vtx)$ and $\rateset$ satisfy
Assumption \ref{mod:WPassu}.
\end{lemma}

The proof of Lemma \ref{lem-ass} is a technical extension of results in \cite{GanRam-Hydro22}, and thus deferred to Appendix \ref{WP}. As shown therein, Assumption \ref{mod:WPassu} in fact holds for the large class of finitely dissociable graphs introduced in \cite[Definition 5.11]{GanRam-Hydro22} whenever the jump rate function family $\rateset$ satisfies Assumptions \ref{mod:regular} and \ref{mod:WPbd}.

\ind 
We now state the generalization of Theorem \ref{res:MRF}. 
\begin{proposition}
\label{gen:MRF}
Suppose the graph $G$, initial data pair $(\vms,\vtx)$ and jump rate function family $\rateset$ satisfies Assumptions \ref{mod:regular}, \ref{mod:WPbd} and \ref{mod:WPassu}. Let $X$ be the solution to the SDE \eqref{mod:infpart}-\eqref{ic} and let $\mdeg\geq 2$ be an integer.
If $\vms$ forms an $\mdeg$-MRF (respy. $\mdeg$-SGMRF) with respect to $G$ then for each $t\in (0,\infty]$, $(\vms,X[t))$ forms an $\mdeg$-MRF (respy. $\mdeg$-SGMRF) with respect to $G$.
\end{proposition}

The proof of Proposition \ref{gen:MRF} is given in Section \ref{MRF:rnd}. Its outline is as follows. Fix $G$, $(\vms,\vtx)$ and $\rateset$ as in the proposition, and let $X = \{X_v\}_{v \in V}$ be a solution to the associated SDE \eqref{mod:infpart}-\eqref{ic}. 
Also, fix $\mdeg \geq 2$ and assume that $\vms$ forms an $\mdeg$-MRF on $\mksp^V$
with respect to $G$. 
For any $t \in (0,\infty)$, our proof that $(\vms,X[t))$ also forms an $\mdeg$-MRF can be broken into four main steps. First, in Section \ref{MRF:ref},
we construct a sequence of $\X^V$-valued ``reference processes'' $\{\Xref^n\}_{n \in\ N},$ with initial data pair $(\vms,\vtx)$ having the property that 
for any partition $A,B, S$ of $V$ for which $S = \kneigh{\mdeg}{A}$ and $A$ is finite, we have
\begin{equation}
\label{MRF:refCI}
(\vms_A,\Xref^n_A[t))\indp (\vms_B,\Xref^n_B[t))| (\vms_S,\Xref^n_S[t)),\quad t \in (0,\infty)
\end{equation}
for all $n$ sufficiently large (depending on $S$). 
Next, in Section \ref{MRF:key} we compute the Radon-Nikodym derivative of the law on path space of the IPS $X$ with respect to that of
the reference process
$\Xref^n$ for any $n \in \N$ by first establishing a duality relation between IPS and point processes (see Proposition \ref{duo:duality}) and then leveraging results from point process theory. Next, combining an explicit factorization of this Radon-Nikodym derivative with 
a slight modification of a result from \cite{PutSch85} (see Lemma \ref{lem:PutSch}), we prove in Section \ref{MRF:rnd} that \eqref{MRF:refCI} must hold with $\Xref^n$ replaced by $X$, which implies that $(\vms,X[t))$ forms an $\mdeg$-MRF with respect to $G$
for all $t \in (0,\infty)$. Finally, we extend the result to the case $t = \infty$ via a standard martingale argument. 
The proof of preservation of the SGMRF property proceeds in a similar fashion, first assuming $\vms$ forms an $\mdeg$-SGMRF and following the above argument, where now $A$ may be infinite, but $S$ must still be finite.

\subsection{Reference Processes and their Conditional Independence Properties}
\label{MRF:ref}

For the remainder of the article, we assume that $\X = \Z$. Note that this is without loss of generality because
the IPS $X$ may be regarded as a process with 
state space $\Z^V$ such that $X_v(t) \in \X$ almost surely for all $v \in V$ and $t \in \R_+$. Given initial data pair $(\vms,\vtx)$ from \eqref{ic}, for any finite vertex set $W\subseteq V$, let $\Xref^W$ be the solution to the SDE
\begin{equation}
\label{rn:infpart}
\Xref^W_v(t) = \begin{cases}
\vtx(\vms_v) + \int_{(0,t]\times (0,1]\times\jmps}j\,\poiss_v(ds,dr,dj) &\te{ if } v \in W,\\
\vtx(\vms_v) + \int_{(0,t]\times \R_+\times\jmps}j\indic{r\leq \rate^v_j(s,\vms,\Xref^W)}\,\poiss_v(ds,dr,dj) &\te{ if } v \notin W. 
\end{cases}
\end{equation}
The equation \eqref{rn:infpart} can be rewritten in terms of the modified rate function family $\wh{\rateset}^W \defeq \{\rateref^{W,v}_j\}_{v \in V,j \in\jmps}$ as follows:
\[\Xref^W_v(t) = \vtx(\vms_v) + \int_{(0,t]\times \R_+\times\jmps}j\indic{r\leq \rateref^{W,v}_j(s,\vms,\Xref^W)}\,\poiss_v(ds,dr,dj), \quad v \in V.\]
Thus, by Assumption \ref{mod:WPassu}, the SDE \eqref{rn:infpart} is strongly well-posed and $\Xref^W$ is a.s. uniquely defined.

\ind Fix an arbitrary vertex $\root \in V$ and define 
\begin{equation}
\Xref^n \defeq \Xref^{V_n} \quad\te{ for }\quad V_n := \{v \in V: d_G(v,\root)\leq n\}, \qquad n \in \N. 
\label{ref:Xndef}
\end{equation}

\ind The main result of this section, Proposition \ref{ref:CI} below,
shows that $(\vms,\Xref^n)$ has a conditional structure that partially resembles an $\mdeg$-SGMRF. 
\begin{proposition}
\label{ref:CI}
Suppose $G,(\vms,\vtx)$ and $\rateset$ satisfy Assumptions \ref{mod:regular}, \ref{mod:WPbd} and \ref{mod:WPassu}, and for each $n \geq 2$ let $\Xref^n$ be as defined in \eqref{rn:infpart} and \eqref{ref:Xndef}. Fix an integer $\mdeg \geq 1$ and let
$A,B, S \subseteq V$ form a partition of $V$ such that $S = \kneigh{\mdeg}{A} \subseteq V_{n-1}$ and also suppose
$\vms_A\indp \vms_B|\vms_{S}$. Then for any $t \in (0,\infty)$, 
\begin{equation}
\label{ref:refCIeq}
(\vms_A,\Xref^n_A[t))\indp (\vms_B,\Xref^n_B[t))|(\vms_{S},\Xref^n_{S}[t)).
\end{equation}
\end{proposition}

The proof of Proposition \ref{ref:CI} is given after the following abstract technical lemma, which provides sufficient conditions under which conditional independence properties can be transferred from one collection of random elements to another. 
\begin{lemma}
\label{ref:CIfact}
For each $i,j = 1, 2, 3$, let $Z^j_i$ be a random element taking values in some Polish space $\Pol^j_i$ that satisfies the following properties: 
\begin{enumerate}
\item $\{Z^2_i\}_{i = 1, 2, 3}$ is a set of mutually independent random elements that is independent of $\{Z^1_i\}_{i=1,2,3}$;
\item $\hfiltm^{Z^1_i}\subseteq\hfiltm^{Z^3_i}\subseteq \hfiltm^{Z^{\{1,2\}}_i}$ for $i=1,2,3,$ where $Z^{\{1,2\}}_i = (Z^1_i,Z^2_i)$.
\end{enumerate}
Then
\begin{equation}
\label{CI-3}
Z^1_1\indp Z^1_2|Z^1_3
\end{equation}
implies $Z_1^3\indp Z_2^3|Z^3_3$.
\end{lemma}

\ind 
Relegating the proof of the lemma to Appendix \ref{ap-condind}, 
we first apply it to prove Proposition \ref{ref:CI}.
\begin{proof}[Proof of Proposition \ref{ref:CI}:]
Fix a partition $A,B,S\subseteq V$ satisfying the conditions of the proposition, and for notational convenience, set $D_1 := A,D_2 := B, D_3 := S$. To prove Proposition \ref{ref:CI}, it clearly suffices to show that the conditions of Lemma \ref{ref:CIfact} are satisfied by $Z^1_i = \vms_{D_i}$, $Z^2_i = \poiss_{D_i}$ and $Z^3_i = (\vms_{D_i},\Xref^n_{D_i}[t))$, $i=1,2,3$, since then Lemma \ref{ref:CIfact} implies that $Z_1^3\indp Z_2^3|Z_3^3$, which is equivalent to \eqref{ref:refCIeq}.

\ind
First, note that $\vms_{D_1}\indp \vms_{D_2}|\vms_{D_3}$ by assumption and so $(Z^1_j)_{j=1,2,3}$ satisfies \eqref{CI-3}. Further, since $D_i$, $i = 1, 2, 3,$ are disjoint,
$\{\poiss_{D_i}\}_{i=1,2,3}$ are mutually independent and, by assumption, also independent of $\{\vms_{D_i}\}_{i=1,2,3}$. Thus, $Z^1$ and $Z^2$ satisfy 
property 1 of Lemma \ref{ref:CIfact}. 
It only remains to verify the measurability condition stated in property 2.
The first inclusion
${\mathcal H}^{Z_i^1} \subseteq {\mathcal H}^{Z_i^3}$ holds trivially for $i = 1, 2, 3$. 
To prove the second inclusion, note that 
$\Xref^n$ is the a.s. unique solution to the SDE
\eqref{rn:infpart} with $W = V_n$ and so substituting the local jump rate functions $\alt{\rateset} \defeq \{\locrate^v_j\}_{v \in V,j\in\jmps}$ from condition 1 of Assumption \ref{mod:regular}, we see that its marginal 
$\Xref^n_A$ on the set $A$ solves the following SDE: 
\begin{equation}
\label{ref:XrefA}
\Xref^n_v(t) = \begin{cases}
\x(\vms_v) + \int_{(0,t]\times (0,1]\times\jmps}j\poiss_v(ds,dr,dj)&\te{ if } v \in V_n\cap A,\\
\x(\vms_v) + \int_{(0,t]\times \R_+\times\jmps}j\indic{r \leq \alt{\rate}^v_j(s,\Xref^n_{\cl{v}},\vms_{\cl{v}})}\poiss_v(ds,dr,dj)&\te{ if } v \in A\setminus V_n.
\end{cases}
\end{equation}
We now claim that $\cl{v}\subseteq A$ for every $v \in A\setminus V_n$,
and thus the SDE \eqref{ref:XrefA} for the marginal $\Xref_A$ is autonomously defined. 
The claim holds because for any $v \in A\setminus V_n$, the fact that $v \in A$, $S = \kneigh{\mdeg}{A} \subseteq V_{n-1}$ and $S\cap B = \emptyset$ imply 
$\cl{v} \cap B \subseteq S\cap B = \emptyset$, whereas the fact that $v \notin V_{n}$ implies $\cl{v} \cap S \subseteq \cl{v}\cap V_{n-1} = \emptyset$. Since $A, B$ and $S$ form a partition of
$V$, this shows that $\cl{v}\subseteq A$. 

\ind
We now show that the SDE \eqref{ref:XrefA} is strongly well-posed. 
Fix any solution space $\fpspace$ supporting the driving Poisson processes $\alt{\poiss}_A$ (in the sense of Definition \ref{mod:drive}) and $\vms_A$, and consider an extension $(\tilde{\Omega}, \tilde{{\mathcal F}}, \tilde{\mathbb{F}}, \tilde{\mathbb{P}})$ of the solution space that 
also supports i.i.d. Poisson processes $\alt{\poiss}_{V\setminus A}$ and $\vms_{V \setminus A}$ such that $\alt{\poiss}\defeq (\alt{\poiss}_v)_{v \in V}$ is a collection of driving Poisson processes of \eqref{rn:infpart} with $W = V_n$.
Let $\alt{Y}^A$ and $\alt{Z}^A$ be two weak solutions to \eqref{ref:XrefA} on the solution space $\fpspace$ with the same initial
data $\vms_A$, and, recalling that \eqref{rn:infpart} is strongly well-posed, let $\alt{X} = \alt{X}^n$ be the a.s. unique solution to \eqref{rn:infpart} with $W = V_n$ and initial data $\vms$. Then, using the property that the marginal on $A$ of any solution to \eqref{ref:XrefA} is autonomously defined, the processes 
\[\alt{Y}_v \defeq \begin{cases}
\alt{Y}^A_v &\te{ if }v \in A\\
\alt{X}_v &\te{ if } v \notin A
\end{cases}\quad \te{and}\quad \alt{Z}_v \defeq\begin{cases}
\alt{Z}^A_v &\te{ if }v \in A\\
\alt{X}_v &\te{ if } v \notin A
\end{cases},\]
are both solutions to \eqref{rn:infpart} on $\fpspace$ with $W = V_n$. Therefore $\alt{X} = \alt{Y} = \alt{Z}$ a.s., and so in particular
$\alt{Y}^A = \alt{Z}^A$ a.s.. Thus, \eqref{ref:XrefA} is strongly well-posed.
In turn, the strong well-posedness of \eqref{ref:XrefA} implies $\Xref^n_A$ is the a.s. unique solution to \eqref{ref:XrefA} driven by $\poiss_A$, and hence by Remark \ref{mod:whystrong}, $(\vms_{A},\Xref^n_{A}[t))$ must be $\hfiltm_{t-}^{\poiss_{A}}\vee\hfiltm^{\vms_{A}}$-measurable. This proves the second inclusion in property 2 for the case $i = 1$. 

\ind 
The case $i = 2$ can be argued similarly. We first claim that $\kgneigh{\mdeg}{B}{G} \subseteq S$. Indeed, if the claim were not true, then since $A,B,S$ is a (disjoint) partition of $V$, 
it must be true that $\kgneigh{\mdeg}{B}{G}\cap A \neq \emptyset$ or equivalently, there must exist $u \in A$ and $v\in B$
such that $d_G(u,v) \leq \mdeg$. However, since $S = \kgneigh{\mdeg}{A}{G}$ by assumption, this implies
$v \in \kgneigh{\mdeg}{A}{G}\cap B = S\cap B =\emptyset$, which is a contradiction. This proves the claim. 
Given the claim, an identical argument as that used for $i=1$ shows that $\Xref^n_B$ is autonomously defined and $(\vms_{B},\Xref^n_{B}[t))$ is $\hfiltm_{t-}^{\poiss_{B}}\vee\hfiltm^{\vms_{B}}$-measurable, thus proving the second inclusion
in property 2 of Lemma \ref{ref:CIfact} when $i = 2$. The proof for the case $i= 3$ is much simpler. By 
\eqref{rn:infpart} with $W = V_n$,
for each $v \in S\subset V_n$, $\Xref^n_v$ is given by the one-dimensional, $\hfiltm^{\vms_S}\vee\hfilt^{\poiss_S}$-adapted process $t\mapsto \x(\vms_v) + \int_{(0,t]\times(0,1]\times\jmps}j\poiss_v(ds,dr,dj)$. This shows that $\Xref^n_v$ is a measurable function of $(\vms_v,\poiss_v)$. Since this is true for each $v \in S$, this proves that
property 2 of Lemma \ref{ref:CIfact} also holds for $i = 3$. 
This completes the verification of the conditions of Lemma \ref{ref:CIfact} for the $\{Z^j_i\}_{i,j\in\{1,2,3\}}$ defined at the start of the proof,
and thus proves the proposition. 
\end{proof}

\subsection{Change of Measure Results on Infinite Graphs}
\label{MRF:key}

In this section, we identify the form of the Radon-Nikodym derivative of the law of the solution to the SDE \eqref{mod:infpart}-\eqref{ic} with respect to that of the reference process $\Xref^W$ for finite $W$. To state our result, we first introduce the notions of proper trajectories and 
their so-called jump characteristics. Recall the definition of $\jmp{x}{t}$ and $\Delta x$ from Section \ref{nota}.

\begin{definition}
\label{mod:proper}
For any $U\subseteq V$, we say a c\`adl\`ag function $x \in \cad^U$ is proper if for every $u \neq v \in U$ and $t < \infty$, $\jmp{x_v}{t} \cap \jmp{x_u}{t} = \emptyset$.
\end{definition}

\begin{definition}
\label{duo:jumpchar}
Fix $U\subseteq V$ finite, $t\in(0,\infty)$, and let $x\in \cad_{t-}^U$ be proper. Then the \emph{jump characteristics} of $x$ are the elements $\{(\thx{k}{x},\jx{k}{x},\vx{k}{x})\}\subset (0,\infty) \times \jmps\times U$, where $\{\thx{k}{x}\} \defeq \jmp{x}{\infty}$ 
is an increasing sequence, and $\Delta x_{\vx{k}{x}}(\thx{k}{x}) = \jx{k}{x}$ for each $k < |\jmp{x}{\infty}|+1$. When the trajectory $x$ is clear from the context, we simply write $\{(\thi{k},\ji{k},\vi{k})\}$ for $\{(\thx{k}{x},\jx{k}{x},\vx{k}{x})\}$.
\end{definition}

\ind We now establish conditions under which the jump characteristics of a process exist. 
\begin{lemma}
\label{proper:a.s.}
Suppose the jump rate function family $\rateset$ satisfies conditions 2 and 3 of Assumption \ref{mod:regular}. Given the initial data pair $(\vms,\vtx)$, for any finite $W\subseteq V$, let $\Xref^W$ be any weak solution to the SDE \eqref{rn:infpart} and let $X$ be any weak solution to \eqref{mod:infpart}-\eqref{ic} for the same initial data pair $(\vms,\vtx)$. Then for any finite $U\subseteq V$, the jump characteristics of $\Xref^W_U$ and $X_U$ are almost surely well defined. 
\end{lemma}
\begin{proof}
Note that $\Xref^{\emptyset}$ is a weak solution to \eqref{mod:infpart}-\eqref{ic} so we may write $X = \Xref^{\emptyset}$ without loss of generality. Thus, for any fixed, finite $U,W\subseteq V$, it suffices to prove that the jump characteristics of $\Xref^W_U$ exist. This occurs precisely when $\Xref^W_U$ is proper and its discontinuity times can be enumerated in increasing order. Suppose $\Xref^W$ solves \eqref{rn:infpart} on the solution space $\fpspace$ supporting the driving Poisson processes $\poiss$. Fix any $u,v \in V$ and $t \in \R_+$. Note that by \eqref{rn:infpart}, 
$\jmp{\Xref^W_v}{t} \subseteq \Emc_v$ for each $v \in V$, 
where $\Emc_v \defeq \{s \in [0,t]: \poiss_v(\{s\}\times \R_+\times \jmps)=1\}$. Thus, for any $u,v \in V$ with $u \neq v$,
\[\jmp{\Xref^W_u}{t}\cap \jmp{\Xref^W_v}{t} \subseteq \Emc_u\cap \Emc_v = \emptyset \te{ a.s.}\]
by the independence of $\poiss_u$ and $\poiss_v$. Therefore $\Xref^W$ is almost surely proper, which implies that $\Xref^W_U$ is proper. Furthermore, because $U$ is finite, $\X^U$ is discrete and equipped with a complete metric. Therefore, $|\jmp{\Xref^W_U}{t}| < \infty$ for every $t \in \R_+$, which shows that the jump characteristics of $\Xref^W_U$ exist.
\end{proof}

\ind
Our next result is a general change of measure result that characterizes the Radon-Nikodym derivative of the law of the solution to the SDE \eqref{mod:infpart}-\eqref{ic} with respect to the reference process $\Xref^W$ in \eqref{rn:infpart} in terms of the jump characteristics of $\Xref^W_W$, which are a.s. well defined by Lemma \ref{proper:a.s.}. 
\begin{proposition}
\label{rn:girs}
Let $G$ be a deterministic (not necessarily locally finite) graph. Suppose the jump rate function family $\rateset \defeq\{\rate^v_j\}_{v\in V,j \in \jmps}$ satisfies conditions 2 and 3 of Assumption \ref{mod:regular}. Let $W\subseteq V$ be any finite set and denote $\mref^W \defeq \law(\vms,\Xref^W)$ where $\Xref^W$ is any weak solution to the SDE \eqref{rn:infpart} for the initial data pair $(\vms,\vtx)$. Also assume that for any $(t,j,v) \in \R_+\times\jmps\times V$, 
\begin{equation}
\label{rn:localint}
\int_0^t \rate^v_j(s,\vms,\Xref^W)\,ds < \infty \te{ a.s..}
\end{equation}
Define the filtration $\nfilt \defeq \hfiltm^{\vms}\vee\hfilt^{\Xref^W}$ and define the process $L^W = (L_t^W)_{t \geq 0}$ as follows: 
\begin{equation}
\label{rn:girsLt}
L^{W}_t \defeq \left[\prod_{0 < \thi{k}\leq t} \rate^{\vi{k}}_{\ji{k}}(\thi{k},\vms,\Xref^W)\right]\exp\left(-\sum_{(j,v) \in \jmps\times W}\int_{(0,t]} \left(\rate^v_j(s,\vms,\Xref^W) - 1\right)\,ds\right),
\end{equation}
where $\{(\thi{k},\ji{k},\vi{k})\}$ are the jump characteristics of $\Xref^W_W$. 
Then $L^{W}$ is a $\nfilt$-local martingale. Moreoever, if $L^{W}$ is also a $\nfilt$-martingale, then there exist a measure $\alt{\mu}\in \P((\mksp\times\cad)^V)$ and a weak solution $X$ to the SDE \eqref{mod:infpart}-\eqref{ic} for the initial data pair $(\vms,\vtx)$ such that
\[ \frac{d\alt{\mu}_{t-}}{d\mref^W_{t-}}(\vms,\Xref^W[t)) = L^{W}_{t-}, \quad t > 0, \mbox{a.s.}, \]
where $\alt{\mu}_{t-}$ and $\mref^W_{t-}$ are the restrictions of the respective measures $\alt{\mu}$ and $\mref^W$ to the space $\borel((\mksp\times \cad_{t-})^V)$, and $\law(\vms,X) = \alt{\mu}$.
\end{proposition}

\ind When $G$ is finite, this result is well known if the IPS is Markov (e.g. \cite[Example 15.2.10]{Bre20}). When $G$ is finite and the IPS is non-Markov, the result can be deduced by combining duality characterizations of the IPS in terms of a point process (as discussed in Section \ref{key:duo}) with standard change of measure theorems for point processes
(e.g., \cite[Theorem 15.2.7]{Bre20}, \cite[Proposition 14.4.III]{DalVer08} or \cite[VIII T10]{Bre81}) and 
simple estimates to prove that the candidate Radon-Nikodym derivative is indeed a martingale (rather than just a local martingale). However, the case when $G$ is infinite is more subtle even for Markov IPS due to the possibility of explosions, 
and Proposition \ref{rn:rnd} in fact addresses the more general
case when the graph may not even be locally finite, necessitating further care. 
This more general setting is 
of interest for the study of analogous properties of IPS on random graphs. In fact, by letting $G$ be the infinite, complete graph, Proposition \ref{rn:girs} becomes applicable to solutions of a more general class of infinite-dimensional Poisson-driven SDEs not necessarily arising as locally interacting processes with
respect to some graph. 
The proof of Proposition \ref{rn:girs} relies on duality characterizations and is hence deferred to Section \ref{rnpf}. 

\begin{remark}
\label{rn:nWP}
Note that Proposition \ref{rn:girs} does not require uniqueness in law of solutions to the SDEs \eqref{mod:infpart}-\eqref{ic} and \eqref{rn:infpart}. It simply establishes a correspondence between the laws of specific weak solutions to the two SDEs when
$L^W$ is a $\nfilt$-Martingale. 
\end{remark}

\ind 
However, in the presence of well-posedness, as a corollary of Proposition \ref{rn:girs}, we obtain
the following change of measure result, which is used to prove Proposition \ref{gen:MRF}. 
\begin{corollary}
\label{rn:rnd}
Suppose $G$, $(\vms,\vtx)$ and $\rateset$ satisfy Assumptions \ref{mod:regular}, \ref{mod:WPbd} and \ref{mod:WPassu}. For any $n \in \N$, 
let $V_n$ be as in \eqref{ref:Xndef}, let $\Xref^n \defeq \Xref^{V_n}$ and $X$ be solutions to \eqref{mod:infpart}-\eqref{ic} and \eqref{rn:infpart}, respectively, and let $\mref^n =\law(\vms,\Xref^n)$ and $\mu = \law(\vms,X)$. Then, for any $t \in (0,\infty),$
\begin{equation}
\label{rn:rngen}
\frac{d\mu_{t-}}{d\mref^n_{t-}}(\vms,\Xref^n[t)) =
\left[\prod_{0 < \thi{k}^n < t} \rate^{\vi{k}^n}_{\ji{k}^n}(\thi{k}^n,\vms,\Xref^n)\right]
\exp\left(-\sum_{(j,v) \in \jmps\times V_n}\int_{(0,t)} \left(\rate^v_j(s,\vms,\Xref^n) - 1\right)\,ds\right) a.s., 
\end{equation}
where $\{(\thi{k}^n,\ji{k}^n,\vi{k}^n)\}_k$ is the set of jump characteristics of $\Xref^n_{V_n}$ in the sense of Definition \ref{duo:jumpchar}.
\end{corollary}
\begin{proof} Fix $n \in \N$ and let $\nfilt \defeq \hfiltm^{\vms}\vee\hfilt^{\Xref^n}$. Applying Proposition \ref{rn:girs} with $W = V_n$, the process
$L^{n} \defeq L^{V_n}$ in \eqref{rn:girsLt} 
is a $\nfilt$-local martingale. We first prove that it is in fact a $\nfilt$-martingale. Let $\{\theta_{\ell}\}_{{\ell} \in \N}$ be a localizing sequence for $L^n$ and for each ${\ell} \in \N$, let $\tau_{\ell} :=\inf\{t: |\jmp{\Xref^n_{V_n}}{t}| \geq {\ell}\}$ where the infimum of an empty set is taken to be infinite. Note that $\tau_{\ell}$ is a $\nfilt$-stopping time. Because $\Xref^n_{V_n}$ is a.s. c\`adl\`ag and $\X^{V_n}$ is a discrete space, $\Xref^n_{V_n}$ can only have finitely many discontinuities in any finite time interval. Thus, $\lim_{{\ell}\to\infty} \tau_{\ell} = \infty$ a.s., so $\{\tau_{\ell}\wedge \theta_{\ell}\}_{{\ell} \in \N}$ is also a localizing sequence for $L^n$. Recall the definition of $C: \N\times\R_+\to\R_+$ from Assumption \ref{mod:WPbd} and let $d := \max\{|\cl{v}|:v \in V_n\}$. Then for each $t \in \R_+$,
\[\sup_{{\ell} \in \N}|L^n_{t\wedge \tau_{\ell}\wedge \theta_{\ell}}| \leq |C(d,t)|^{{\ell}\wedge |\jmp{\Xref^n_{V_n}}{t}|}\exp\left(t|V_n||\jmps|\right) \leq |C(d,t)|^{|\jmp{\Xref^n_{V_n}}{t}|}\exp\left(t|V_n||\jmps|\right).\]
However, note that $|\jmp{\Xref^n_{V_n}}{t}|\sim \te{Poiss}(|V_n||\jmps|t)$ because by \eqref{mod:refrate} $\rateref^{V_n,v}_j(\cdot,\vms,\Xref^n) = 1$ whenever $(j,v) \in \jmps\times V_n$. Since $\jmps$ is assumed to be finite, this implies
\[\ex{\sup_{{\ell} \in \N} |L^n_{t\wedge \tau_{\ell}\wedge \theta_{\ell}}|} \leq \exp\left(t|V_n||\jmps|\right)\ex{|C(d,t)|^{|\jmp{\Xref^n_{V_n}}{t}|}} < \infty,\]
which shows that for each $t \geq 0$, the sequence $\{L^n_{t\wedge \tau_{\ell}\wedge \theta_{\ell}}\}_{{\ell} \in \N}$ is dominated by an integrable random variable, and is therefore uniformly integrable. By \cite[Proposition 1.8]{ChuWil14}, it follows that $L^n$ is a $\nfilt$-martingale and for any $t \in \R_+$, $\exnd{L^n_t} = 1$. Proposition \ref{rn:girs} then implies that the measure $\alt{\mu} \in \P((\mksp\times\cad)^V)$ defined by
$d\alt{\mu}_{t-}/d\mref^{n}_{t-}(\vms,\Xref^n) = L^n_{t-}$ a.s. for all $t > 0$ is the law of a weak solution $\alt{X}$ to \eqref{mod:infpart}-\eqref{ic} for the initial data $\vms$ on the graph $G$. By Assumption \ref{mod:WPassu}, the SDE \eqref{mod:infpart}-\eqref{ic} is
well-posed, which implies that $(\vms,\alt{X}) \deq (\vms,X)$.
This shows 
$d\mu_{t-}/d\mref^n_{t-}(\vms,\Xref^n) = d\alt{\mu}_{t-}/d\mref^n_{t-}(\vms,\Xref^n) = L^n_{t-}$ a.s. for every $t > 0$, as desired.
\end{proof}

\subsection{Proof of Proposition \ref{gen:MRF}}
\label{MRF:rnd}

As mentioned earlier, the proof of the preservation of the MRF property over any finite time interval proceeds by transfering analogous conditional 
independence properties for the reference processes established in Proposition \ref{ref:CI} to the original
IPS. This makes use of a modified version of a result from \cite{PutSch85}
stated in Lemma \ref{lem:PutSch} below. 
\begin{lemma} 
\label{lem:PutSch}
Let $(\mspacenew)$ be a measurable space, let 
$\pfiltnew_i \subset \pfiltnew$, $i = 0, 1, 2$, be sub-$\sigma$-algebras, and let
$\Pnew_0$ and $\Pnew_1$ be two probability measures on
$(\mspacenew)$ such that $\Pnew_1 \ll \Pnew_0$.
Assume that under $\Pnew_0$, $\pfiltnew_1$ and $\pfiltnew_2$ are
conditionally independent given $\pfiltnew_0$. If in addition, the Radon-Nikodym derivative $\RNnew \defeq d\Pnew_1/d\Pnew_0$ with respect to
$\vee_{i=1}^3\filtnew_i$ satisfies $\RNnew = \RNnew_1 \RNnew_2$ almost surely for some $\pfiltnew_i\vee\pfiltnew_0$-measurable random variables $\RNnew_i$, $i = 1, 2,$ then under $\Pnew_1$, $\pfiltnew_1$ and $\pfiltnew_2$ are also 
conditionally independent given $\pfiltnew_0$.
\end{lemma} 
\begin{proof}
If the filtrations were required to be complete, then this result would follow from \cite[Theorem 3.6]{PutSch85} under the stronger assumption that $\Pnew_1$ is equivalent to $\Pnew_0$. However, as elaborated below, the same argument used in \cite{PutSch85} also shows that the result holds under the weaker assumption of a not necessarily complete filtration and only absolute continuity of $\Pnew_1$ with respect to $\Pnew_0$ (rather than equivalence). Let $Z$ be any bounded, $\pfiltnew_1$-measurable random variable. Then, applying \cite[Proposition B.41]{Bjo20} and using the $\Pnew_0$-conditional independence of $\pfiltnew_1$ and $\pfiltnew_2$ given $\pfiltnew_0$ and the $\pfiltnew_i$-measurability of $\RNnew_i$, $i=1,2,$ we have $\Pnew_1$-a.s., 
\begin{align*}
\exmu{\Pnew_1}{Z|\pfiltnew_0\vee\pfiltnew_2} &= \frac{\exmu{\Pnew_0}{Z\RNnew_1\RNnew_2|\pfiltnew_0\vee\pfiltnew_2}}{\exmu{\Pnew_0}{\RNnew_1\RNnew_2|\pfiltnew_0\vee\pfiltnew_2}} = \frac{\RNnew_2\exmu{\Pnew_0}{Z\RNnew_1|\pfiltnew_0\vee\pfiltnew_2}}{\RNnew_2\exmu{\Pnew_0}{\RNnew_1|\pfiltnew_0\vee\pfiltnew_2}} = \frac{\exmu{\Pnew_0}{Z\RNnew_1|\pfiltnew_0}}{\exmu{\Pnew_0}{\RNnew_1|\pfiltnew_0}}. 
\end{align*}
Thus, $\exmu{\Pnew_1}{Z|\pfiltnew_0\vee\pfiltnew_2}$ is $\pfiltnew_0$-measurable and therefore $\Pnew_1$-a.s. equal to $\exmu{\Pnew_1}{Z|\pfiltnew_0}$. Since $Z$ is arbitrary, $\pfiltnew_1$ and $\pfiltnew_2$ are conditionally independent given $\pfiltnew_0$ under $\Pnew_1$ as well.
\end{proof}

\begin{proof}[Proof of Proposition \ref{gen:MRF}:] Assume $\vms$ forms an $\mdeg$-MRF (respy, $\mdeg$-SGMRF) for some $\mdeg \geq 2$.
Fix $n \in \N$. By Assumption \ref{mod:WPassu}, the SDE \eqref{mod:infpart}-\eqref{ic} and the SDE
\eqref{rn:infpart} with $W = V_n$, are both strongly well-posed. Let 
$X$ and $\Xref^n$ be the respective solutions. For $t > 0$, let $\mu_{t-} \defeq \law(\vms,X[t))$, $\mref^n_{t-} \defeq \law(\vms,\Xref^n[t))$, $\mu \defeq \mu_{\infty-}$ and $\mref^n \defeq \mref^n_{\infty-}$.
Suppose that $A,B,S\subseteq V$ partition $V$ with $|A|<\infty$ (respy, $|S| <\infty$ if $\vms$ is an $\mdeg$-SGMRF) and let $n$ be sufficiently large so that $S = \kneigh{\mdeg}{A} \subseteq V_{n-1}$. By assumption, we have $\vms_A \indp \vms_B|\vms_S$ and so if we fix
$t > 0$, it follows from Proposition \ref{ref:CI} that under $\mref^n_{t-}$,
$\pfiltnew_1 := \borel((\mksp\times\cad_{t-})^A)$
is independent of $\pfiltnew_2 := \borel((\mksp\times\cad_{t-})^B)$ given 
$\pfiltnew_0 := \borel((\mksp\times\cad_{t-})^S)$.

\ind To transfer this to $\mu_{t-}$, we would like to apply Lemma \ref{lem:PutSch} with the following substitutions: 
$(\mspacenew) = \left((\mksp\times\cad_{t-})^V,\borel((\mksp\times\cad_{t-})^V))\right)$, $\pfiltnew_i$, $i = 0, 1, 2$ as just defined,
$\Pnew_0 = \mref^n_{t-}$, $\Pnew_1 = \mu_{t-}$, $\RNnew = d\mu_{t-}/d\mref^n_{t-}$. Also, note that by
Corollary \ref{rn:rnd}, 
$\Pnew_1 \ll \Pnew_0$ on $\borel((\mksp\times\cad_{t-})^V)$ with Radon-Nikodym derivative $\RNnew\defeq d\mu_{t-}/d\mref^n_{t-}$. To verify the remaining condition of Lemma \ref{lem:PutSch}, it only remains to show that $\RNnew$ factorizes in the right way. 
To this end, recall the definition of the local jump rate function family $\alt{\rateset} \defeq\{\locrate^v_j\}_{v \in V,j\in\jmps}$ from Assumption \ref{mod:regular},
and for $v \in V$, define $g_v: (\mksp\times\cad_{t-})^{\gcl{v}{G}} \to \R_+$ as follows:
\[ g_v(\dvms,x) := \left[\prod_{0<\thx{k}{x_v}< t} \locrate_{\jx{k}{x_v}}^v(\thx{k}{x_v},\dvms,x)\right]\exp\left(- \sum_{j \in\jmps}\int_{(0,t)} (\locrate^v_j(t,\dvms,x) - 1)\,ds\right) \]
if $x$ is proper, and set $g_v(\dvms,x) = 0$ for any other $x \in \cad_{t-}$.
Letting $\{(\thi{k}^n,\ji{k}^n,\vi{k}^n)\}_{k \in \mathbb{N}}$ denote the jump characteristics of $\Xref^n_{V_n}$ (which are
a.s. well defined by Lemma \ref{proper:a.s.}), Corollary \ref{rn:rnd} and the condition 1 of Assumption \ref{mod:regular} imply 
\begin{align}
\label{rnd:dmudmudef}
\frac{d\mu_{t-}}{d\mref^n_{t-}}
\left(\vms,\Xref^n[t)\right)&=\prod_{v \in V_n}g_v\left(\vms_{\cl{v}},\Xref^n_{\cl{v}}[t)\right) \te{ a.s..}
\end{align}
$\frac{d\mu_{t-}}{d\mref^n_{t-}}:(\mksp\times\cad_{t-})^V\to \R_+$ may be defined for each $(\dvms,x)$ as
\begin{align}
\frac{d\mu_{t-}}{d\mref^n_{t-}}(\dvms,x) &\defeq \prod_{v \in V_n}g_v(\dvms_{\cl{v}},x_{\cl{v}})\nonumber\\
&=\left(\prod_{v \in \cl{A} \cap V_n} g_v\left(\dvms_{\cl{v}},x_{\cl{v}}\right)\right)\left(\prod_{u \in V_n\setminus \cl{A}} g_u\left(\dvms_{\cl{u}},x_{\cl{u}}\right)\right).
\label{rnd:factor}
\end{align}
Since $\mdeg \geq 2$ and $\kgneigh{\mdeg}{A}{G}= S$, the term in the first bracket on the right-hand side of \eqref{rnd:factor} depends only on $(\vms_{A\cup S},\Xref^n_{A\cup S}[t))$
and is thus $\borel((\mksp\times\cad_{t-})^{A\cup S} = \pfiltnew_1\vee\pfiltnew_0$-measurable while the term in the
second bracket depends only on $(\vms_{B\cup S},\Xref^n_{B\cup S}[t))$ and is thus $\borel((\mksp\times\cad_{t-})^{B\cup S} = \pfiltnew_2\vee\pfiltnew_0$-measurable. This completes the verification of all the conditions of Lemma \ref{lem:PutSch}, which allows us to conclude that
\begin{equation}
\label{rnd:finTime}
(\vms_{A},X_{A}[t))\indp (\vms_{B},X_{B}[t))|(\vms_S,X_S[t)). 
\end{equation}
Thus, we have shown that $(\vms,X[t))$ forms an $\mdeg$-MRF (respy. $\mdeg$-SGMRF) with respect to $G$.

\ind We now extend this to the infinite time interval to show that the same is true for $(\vms,X)$.
The argument we use is similar to the one applied in the proof of \cite[Theorem 2.4]{LacRamWuMRF21}. Fix an element $a \notin \X$. For any $x \in \cad$ and $t \in \R_+$, we may embed the truncated function $x[t)$ in $\cad^a \defeq \cad(\R_+,\X\cup\{a\})$ by setting $x[t)(s) = a$ whenever $s \geq t$ and $x[t)(s) = x(s)$ when $s < t$. Thus, $x[t)\to x$ as $t \to \infty$ in $\cad^a$. Let $A,B,S\subseteq V$ be a partition such that $S = \kneigh{\mdeg}{A}$ is finite, and if $\mu_\vms$ is not an $\mdeg$-SGMRF, then assume $A$ is also finite. For each $U \subset V$, define $f_U: (\mksp\times\cad^a)^U\to \R$ to be bounded and continuous. Then, as justified below the display, we have 
\begin{align*}
&\ex{f_A(\vms_A,X_A)f_B(\vms_B,X_B)f_S(\vms_S,X_S)}\\
&\ind= \lim_{s_2\to\infty}\lim_{s_1\to\infty}\ex{f_A(\vms_A,X_A[s_2))f_B(\vms_B,X_B[s_2))f_S(\vms_S,X_S[s_1))}\\
&\ind=\lim_{s_2\to\infty}\lim_{s_1\to\infty}\ex{\ex{f_A(\vms_A,X_A[s_2))|\vms_S,X_S[s_1)}\ex{f_B(\vms_B,X_B[s_2))|\vms_S,X_S[s_1)}f_S(\vms_S,X_S[s_1))}\\
&\ind=\lim_{s_2\to\infty}\ex{\ex{f_A(\vms_A,X_A[s_2))|\vms_S,X_S}\ex{f_B(\vms_B,X_B[s_2))|\vms_S,X_S}f_S(\vms_S,X_S)}\\
&\ind=\ex{\ex{f_A(\vms_A,X_A)|\vms_S,X_S}\ex{f_B(\vms_B,X_B)|\vms_S,X_S}f_S(\vms_S,X_S)},
\end{align*}
where the first equality uses the bounded convergence theorem and continuity of the functions $f_A,f_B,f_S$,
the second equality uses the relation \eqref{rnd:finTime} with $t = s_1$, the third equality
uses $\sigma(\vms_S,X_S) = \vee_{s \in \R_+} \sigma(\vms_S,X_S[s))$, the Doob martingale convergence theorem, the continuity of $f_S$ and the bounded convergence theorem, and the final equality holds due to the bounded convergence theorem and the continuity of $f_A$ and $f_B$. 
This proves $(\vms,X)$ forms an $\mdeg$-MRF (respy. $\mdeg$-SGMRF) with respect to $G$, as desired. 
\end{proof}

\section{Proof of the Radon-Nikodym Derivative Characterization}
\label{sec:rn:rnd}

The goal of this section is to prove the change of measure result in
Proposition \ref{rn:girs}. The proof, which is given in Section \ref{rnpf}, 
relies on a key duality characterization of the IPS that is first established in Proposition \ref{duo:duality} of Section \ref{key:duo}.

\subsection{Dual Processes}
\label{key:duo}

We start in Section \ref{duo:pp} by introducing some standard notation related
to point processes. 

\subsubsection{Point Processes}
\label{duo:pp}

Fix a filtered probability space $\fpspace$ supporting a filtration $\nfilt\subseteq \filt$ and a Polish space $\Pol$. Recall from Section \ref{nota} that $\Nms(\Pol)$ is the (Polish) space of locally finite, nonnegative integer-valued measures on $\Pol$, equipped with the weak-hash topology. A point process $P$ on $\Pol$ is a random element taking values in $\Nms(\Pol)$. For every bounded set $B \subseteq \Pol$, there exists a finite set of points $\{z_i\}_{i=1}^N \subseteq B$ such that $\rpp(\{z_i\}) > 0$ for all $i=1,\dots, N$, and $\rpp(B \setminus \{z_i\}_{i=1}^N) = 0$. These points are called \emph{events}. A point process $P$ on $\Pol$ is \emph{simple} if $\sup_{z \in \Pol} \rpp(\{z\}) \in \{0,1\}$ almost surely. For $\wh{\Pol} \defeq \R_+\times\Pol$, a point process $\rpp$ on $\wh{\Pol}$ is referred to as a \emph{marked} point process on $\R_+$ with marks in $\Pol$. If $\rpp$ has events $\{(t_i,\rho_i)\}_{i=1}^N\subset \wh{\Pol}$, then $\{\rho_i\}_{i=1}^N$ are said to be the \emph{marks} of $\rpp$. If $\rpp([0,T]\times\Pol) < \infty$ a.s. for all $T \in \R_+$, then $\rpp$ is said to be \emph{non-explosive} or \emph{locally finite}. All point processes $\rpp$ that we consider
will be \emph{simple} in the sense that $\sup_{t \in \R_+}\rpp(\{t\}\times \Pol) \in \{0,1\}$ almost surely. A marked point process $\rpp$ on $\wh{\Pol}$ is said to be $\nfilt$-adapted if for every $t \in \R_+$ and $A \in \borel([0,t]\times \Pol)$, $\rpp(A)$ is $\nfiltm_t$-measurable. For any point process $P$ on $\wh{\Pol}$, $\hfilt^P$ is defined to be the minimal filtration satisfying the usual conditions such that $P$ is $\hfilt^P$-adapted. 

\ind Suppose that $\Pol$ is equipped with a nonnegative, locally finite \emph{reference measure} $\ell$. Then a random function $\Gamma: \Omega\times \R_+\times \Pol \to \R_+$ is said to be $\nfilt$-\emph{mark predictable} if it is measurable with respect to $\Pms(\nfilt)\otimes \borel(\Pol)$, where $\Pms(\nfilt)$ is the predictable $\sigma$-algebra generated by $\nfilt$ \cite[page 379]{DalVer08}. As an immediate consequence, it follows that $t\mapsto \int_{z \in A} \Gamma(t,z)\,\ell(dz)$ is $\nfilt$-predictable for all $A \in\borel(\Pol)$. The $\nfilt$-\emph{intensity} of $\rpp$ with respect to the measure $\ell$ is a $\nfilt$-mark predictable process $\Gamma$ such that for each bounded $A \subseteq \Pol$, the process $t \mapsto \rpp([0,t]\times A) - \int_{z \in A}\int_0^t\Gamma(s,z)\,ds\,\ell(dz)$ is a $\nfilt$-local martingale \cite[Definitions 14.1.I,14.3.I]{DalVer08}. 

\subsubsection{Dual Characterizations}
\label{duo:duo}

Recalling the Definition \ref{duo:jumpchar} of the jump characteristics $\{(\thx{k}{x},\jx{k}{x},\vx{k}{x})\}$
of any proper trajectory $x$, we now define the notion of a dual to $x$. 
\begin{definition}
\label{duo:dualdef}
For any (not necessarily finite) $U \subseteq V$, and any proper $x \in \cad^U$, the \emph{dual} of $x$ is the simple marked point process $p \in \Nms(\R_+\times \jmps\times U)$ with event times $\bigcup_{v \in U}\{\thx{k}{x_v}\}$ and marks $\bigcup_{v \in U}\{(\jx{k}{x_v},\vx{k}{x_v})\}$.
\end{definition}
It is easy to see that the duals of weak solutions to \eqref{mod:infpart}-\eqref{ic} generally exist, see Lemma \ref{duo:dualexist} for a complete proof. 
In what follows, $\Sm{\Smc}$ denotes the counting measure on any countable set $\Smc$.

\begin{proposition}
\label{duo:duality}
Suppose that $G = (V,E)$ is a deterministic (but not necessarily locally finite) graph, and 
the jump rate function family $\rateset \defeq \{\rate^v_j\}_{v\in V,j\in\jmps}$ satisfies conditions 2 and 3 of Assumption \ref{mod:regular} and the
local integrability condition in \eqref{rn:localint}. 
Let $X$ be any weak solution to the SDE \eqref{mod:infpart}-\eqref{ic} for the given initial data pair $(\vms,\vtx)$ on the solution space $\fpspace$ and define $\nfilt \defeq \hfiltm^{\vms} \vee\hfilt^{X}\subseteq \filt$. Then for any deterministic (not necessarily finite) subset $U \subseteq V$, the following properties are satisfied:
\begin{enumerate}[(a)]
\item The dual $P_U$ of $X_U$ has a $\nfilt$-intensity on $\jmps\times U$ with respect to the reference measure $\Sm{\jmps\times U}$ that is given explicitly by 
\begin{equation}
\label{duo:int}
\Lambda_U(t,\rho) \defeq \rate^v_j(t,\vms,X) \quad \te{ for }\quad \rho = (j,v) \in \jmps\times U, \te{ }t > 0. 
\end{equation}
\item Let $\alt{X}$ be an a.s. proper $\X^V$-valued c\`adl\`ag process and let $\alt{\vms}$ be a $\mksp^V$-valued random element, defined on a common probability space 
$\pspace$ and 
such that $\alt{X}_v(0) = \vtx(\alt{\vms}_v)$ for each $v \in V$. Set $\alt{\nfilt} \defeq \hfiltm^{\alt{\vms}}\vee\hfilt^{\alt{X}}$. 
If the dual of $\alt{X}_{U}$ is a simple marked point process $\alt{P}_{U}$ with mark space $\jmps\times U$ and $\alt{\nfilt}$-predictable intensity $\alt{\Lambda}_{U}$ with respect to the reference measure $\Sm{\jmps\times U}$ given by 
\begin{equation}
\label{duo:altint}
\alt{\Lambda}_{U}(t,\rho) \defeq \rate^v_j(t,\alt{X},\alt{\vms}) \quad \te{ for }\quad \rho = (j,v) \in \jmps\times U,\te{ } t > 0, 
\end{equation}
then it is possible to extend the probability space $\pspace$ to support a filtration $\wh{\nfilt}\supseteq \alt{\nfilt}$ satisfying the usual conditions and a collection of i.i.d. $\wh{\nfilt}$-Poisson processes $\alt{\poiss} \defeq \{\alt{\poiss}_v\}_{v \in U}$ on $\R^2\times\jmps$ with intensity measure $\leb^2\times\Sm{\jmps}$ such that $\alt{X}$ satisfies 
\begin{equation}
\label{SDE-dual}
\alt{X}_v(t) = \alt{X}_v(0) + \int_{(0,t]\times \R_+\times \jmps} j \indic{r \leq \rate^v_j(s,\alt{X},\alt{\vms})}\,\alt{\poiss}_v(ds,dr,dj) \te{ for } v \in U.
\end{equation}
\end{enumerate}
\end{proposition}

\ind
When $G$ is finite and the IPS is Markov, duality results analogous to Proposition \ref{duo:duality} are well known (e.g. \cite[Example 10.3(a)]{DalVer08}). When $G$ is finite and the IPS is non-Markov, both parts of Proposition \ref{duo:duality} follow easily from scaling arguments (applied to the driving Poisson processes of the SDE \eqref{mod:infpart}-\eqref{ic}) that allow the applicatoin of standard Poisson embedding theorems such as \cite[Theorems 15.3.3 and 15.3.4]{Bre20}. However, when $U$ and $G$ are infinite, the dual point processes $P_U$ and $\alt{P}_U$ of $X_U$ and $\alt{X}_U$ respectively may be explosive, in which case standard embedding theorems cannot be applied directly (although see \cite[Section XIV.4]{Jac79}). In the proof of Proposition \ref{duo:duality}(a), this is easily resolved by noting that for all finite $W\subset U$, $\Lambda_W$ is the restriction of $\Lambda_U$ to $(0,\infty)\times \jmps\times W$. The argument above can then be applied to $\Lambda_W$. However, in the proof of Proposition \ref{duo:duality}(b), a subtlety arises. For each finite $W\subseteq U$, it is possible to construct Poisson processes $\alt{\poiss}^W\defeq \{\alt{\poiss}_v^W\}_{v \in W}$ such that \eqref{SDE-dual} holds when $U$ is replaced by $W$ by appealing to standard Poisson embedding theorems. However, it is not clear how to piece together the processes $\alt{\poiss}^W$ for different finite $W\subseteq U$ to construct a collection $\wh{\nfilt}$ of i.i.d. processes that are Poisson with respect to a \emph{common filtration} $\wh{\nfilt}$. Instead, to circumvent this problem, we provide an explicit construction of the different collections of i.i.d. Poisson processes, $\alt{\poiss}^W$ for finite $W$, with respect to a common filtration. Such explicit constructions appear to be available only for unmarked point processes \cite[Theorem 15.3.4]{Bre20}, \cite[Lemma 4]{BreMas96}. For marked point processes, partial results can be found in \cite[Exercise 14.7.I and Proposition 14.7.I(b)]{DalVer08} but without proof or with the stringent condition that the point process intensity is adapted to the natural filtration. However, in our setting, the intensity of $\alt{P}_W$ is not adapted to its natural filtration.
In order to provide a fully rigorous argument and be self-contained, we include a complete proof below.
\begin{proof}[Proof of Proposition \ref{duo:duality}:] 
\textbf{Proof of (a):} Let $\poiss$ denote the driving Poisson processes associated with the solution $X$, and define $\ov{\nfilt} \defeq \nfilt \vee \hfilt^{\poiss_{U}}$. Let $W\subseteq U$ be finite. Define the point process $\ov{\poiss}^W$ on $\R^2_+\times \jmps\times W$ by
\[\ov{\poiss}^W(\{(t,r,j,v)\}) \defeq \poiss_v\left(\left\{\left(t,\frac{r}{|\jmps||W|},j\right)\right\}\right), \quad (t,r,j,v) \in \R_+^2\times\jmps\times W.\]
Since $\ov{\nfilt}\subseteq \filt$ and $\{\poiss_v\}_{v \in W}$ are i.i.d. $\filt$-Poisson processes that are also $\ov{\nfilt}$-adapted (by definition), it follows that they are also i.i.d. $\ov{\nfilt}$-Poisson processes.
Thus $\ov{\poiss}^W$ is also a $\ov{\nfilt}$-Poisson process on $\R_+^2\times \jmps\times W$. Recalling our assumption that $\jmps$ is finite, let $Q_W = \Sm{\jmps\times W}/(|\jmps||W|)$
be the uniform distribution on $\jmps \times W$. For any $\ov{A} \in \borel(\R_+^2)$, let $\ov{A}' \defeq \left\{\left(t,\frac{r}{|\jmps||W|}\right): (t,r) \in \ov{A}\right\}$ and for $B \subseteq \jmps \times W$ and $v \in W$, define $B_v \defeq \{j\in \jmps: (j,v) \in B\}$. Then $B = \cup_{v \in W} B_v\times \{v\}$ and
\[\ex{\ov{\poiss}^W(\ov{A}\times B)} = \ex{\sum_{(j,v) \in B} \poiss_v(\ov{A}'\times \{j\})} = \frac{\leb^2(\ov{A})|B|}{|\jmps||W|} = \leb^2(\ov{A})Q_W(B),\]
which implies that $\ov{\poiss}^W$ has intensity measure $\leb^2\otimes Q_W$. 

\ind Let $P_W$ be the dual of $X_W$, and note from Definition \ref{duo:dualdef} that $P_W = P_U|_{\R_+\times\jmps\times W}$.
Then the SDE \eqref{mod:infpart}-\eqref{ic} and the definitions of $P_W$ and $\ov{\poiss}^W$ imply that for any $A \in \borel(\R_+)$ and $B \subseteq \jmps\times W$,
\[P_W(A \times B) = \int_{\R_+^2\times\jmps\times W} \indic{s \in A} \indic{(j,v) \in B}\indic{r \leq |\jmps||W|\rate^v_j(s,\vms,X)}\,\ov{\poiss}^W(ds,dr,dj,dv).\]
Since this form coincides with Equation (15.45) of \cite{Bre20}, it follows from \cite[Theorem 15.3.3]{Bre20} that $P_W$ has $\ov{\nfilt}$-intensity
\[\Lambda_W(t,j,v) = |\jmps||W|\rate^v_j(t,\vms,X)Q_W(\{(j,v)\}) = \frac{|\jmps||W|\rate^v_j(t,\vms,X)}{|\jmps||W|} = \rate^v_j(t,\vms,X).\]
Given that $P_U|_{\R_+\times \jmps\times W} = P_W$ for every finite $W\subseteq U$, it follows that $P_U$ has $\ov{\nfilt}$-intensity $\Lambda_{U}$ where for every $(t,j,v) \in \R_+\times\jmps\times U$, 
\[\Lambda_U(t,j,v) = \Lambda_{\{v\}}(t,j,v) = \rate^v_j(t,\vms,X).\]
By condition 3 of Assumption \ref{mod:regular}, for each $(j,v) \in \jmps\times U$, $\Lambda_U(\cdot,j,v)$ has a.s. c\`agl\`ad trajectories and by condition 2 of Assumption \ref{mod:regular} it is $\nfilt$-adapted. Thus, $\Lambda_U$ is $\nfilt$-mark predictable
(in the sense defined in Section \ref{duo:pp}) and hence, $\Lambda_U$ is also the $\nfilt$-intensity of $P_U$. 

\skipLine
\textbf{Proof of (b):}
Let $\alt{X}$ and $\alt{P}_U$ be as stated in the proposition.
The proof 
proceeds via three steps. In step 1, we construct a candidate point process $\alt{\poiss}$ on $\R^2_+\times \jmps\times U$ for the Poisson embedding which we will use to construct $\{\alt{\poiss}_v\}_{v \in U}$. In step 2, we prove that, for an explicitly constructed filtration $\wh{\nfilt}$, $\alt{\poiss}$ is a $\wh{\nfilt}$-Poisson point process on $\R_+^2\times\jmps\times U$. Finally, in step 3, we construct the i.i.d. $\wh{\nfilt}$-Poisson processes $\{\alt{\poiss}_v\}_{v \in U}$ and prove that \eqref{SDE-dual} holds.

\skipLine
\textbf{Step 1: Construct a candidate Poisson process $\alt{\poiss}$.}
\skipLine

Let $\{\wh{\poiss}_v\}_{v \in U}$ be a collection of i.i.d. Poisson processes independent of $\alt{\nfiltm}_\infty$ with intensity measure $\leb^2\otimes \Sm{\jmps}$. Also, define the collection of i.i.d. uniform $[0,1]$-random variables $\{R^v_k\}_{k \in \N,v \in U}$ to be independent of $\alt{\nfiltm}_\infty$ and $\{\wh{\poiss}_v\}_{v \in U}$. Additionally, for each $v \in U$, let $\{(t^v_k,j^v_k)\}_{k \in \N}$ be the collection of events in $\alt{P}_U(dt,dj,\{v\})$ ordered so that $\{t^v_k\}$ is strictly increasing. Since $\alt{P}_U$ is the dual of an a.s. c\`adl\`ag proper process, $\alt{P}_U(\cdot\times\jmps\times\{v\})$ a.s. has finitely many events in any finite interval, and the $\{t_k^v\}$ can be ordered to be strictly increasing with the caveat that if $\alt{P}_U(\R_+\times\jmps\times\{v\}) = K < \infty$, then $t^v_k = \infty$ for all $k > K$. We now define $\alt{\poiss}$ to be the following point process on $\R_+^2\times\jmps\times U$: for any collections of Borel measurable subsets $\{A_v\}_{v \in U}$ of $\R_+^2$ and $\{B_v\}_{v \in U}$ of $\jmps$, and $C \defeq \bigcup_{v \in U} (A_v\times B_v\times \{v\})$,
\begin{align}
\alt{\poiss}(C)& \defeq \sum_{v \in U} \alt{\poiss}_v(A_v\times B_v)\nonumber\\
& \defeq \sum_{v \in U}\int_{A_v\times B_v} \indic{r > \rate^v_j(s,\alt{X},\alt{\vms})}\,\wh{\poiss}_v(ds,dr,dj) + \sum_{v \in U}\sum_{k \in \N} \indic{\left(t^v_k, R^v_k\rate^v_{j^v_k}(t^v_k,\alt{X},\alt{\vms}),j^v_k\right) \in A_v\times B_v}.\label{dualitypf:poissdef}
\end{align}
In the event that $t^v_k = \infty$ for some $v \in U$ and $k \in \N$, $\{t^v_k\}\times\R_+\cap A_v=\emptyset$ because $A_v \in \borel(\R_+^2)$, so \eqref{dualitypf:poissdef} is still well defined. This concludes step 1.

\ind Next, let $P^{R,J}$ be the point process on $\R_+\times[0,1]\times \jmps \times U$ with events $\{(t^v_k,R^v_k,j^v_k,v)\}_{k \in \N,v \in U}$ and define 
\[\wh{\nfilt}\defeq \alt{\nfilt}\vee\hfilt^{\wh{\poiss}}\vee\hfilt^{P^{R,J}}.\]
Note that $\wh{\nfilt}$ satisfies the usual conditions. 

\skipLine 
\textbf{Step 2: Show that $\alt{\poiss}$ is a $\wh{\nfilt}$-Poisson point process.}
\skipLine

Let $H: \R_+^2\times \jmps\times W \to\R_+$ be a nonnegative, $\wh{\nfilt}$-mark predictable random function that is left-continuous (with respect to its first input). Then note that $(t,r,j,v)\mapsto \indic{r > \rate^v_j(t,\alt{X},\alt{\vms})}H(t,r,j,v)$ is also $\wh{\nfilt}$-mark predictable. Lastly, note that because $\wh{\poiss}_W \defeq \{\wh{\poiss}_v\}_{v \in W}$ is independent of $\alt{\nfiltm}_\infty$ and $\{R^v_k\}_{v,k}$, it is also a collection of i.i.d. $\wh{\nfilt}$-Poisson processes with intensity measure $\leb^2\otimes\Sm{\jmps}$. Then by \cite[Theorem 15.1.22]{Bre20}, it follows that
\begin{align}
&\ex{\sum_{v \in W}\int_{\R_+^2\times\jmps} \indic{r > \rate^v_j(s,\alt{X},\alt{\vms})}H(s,r,j,v)\,\wh{\poiss}_v(ds,dr,dj)} \nonumber\\
&\ind\ind = \ex{\int_{\R_+^2\times\jmps\times W} \indic{r > \rate^v_j(s,\alt{X},\alt{\vms})}H(s,r,j,v)\,ds\,dr\,\Sm{\jmps\times W}(dj,dv)}.\nonumber
\end{align}
Let $P^{R,J}_W\defeq P^{R,J}|_{\R_+\times [0,1]\times\jmps\times W}$ and note that $P^{R,J}_W$ is a $\wh{\nfilt}$-adapted non-explosive point process. Furthermore, because $P^{R,J}_W(\cdot,[0,1],\cdot,\cdot) = \alt{P}_W(\cdot,\cdot,\cdot)$ and $\{R^v_j\}_{v \in U,j \in \jmps}$ are i.i.d. and independent of $\alt{P}_W$ with density 1, $P^{R,J}_W$ has $\wh{\nfilt}$-intensity $\Lambda^{R,J}(t,\rho,j,v)\defeq \rate^v_j(t,\alt{X},\alt{\vms})$ with respect to the reference measure $\leb\otimes \leb|_{[0,1]}\otimes \Sm{\jmps\times W}$. The definition of $P^{R,J}_W$ and \cite[Theorem 15.1.22]{Bre20} then imply 
\begin{align}
\notag 
&\ex{\sum_{k \in \N,v \in W}H(t^v_k,R^v_k\rate^v_{j^v_k}(t^v_k,\alt{X},\alt{\vms}),j^v_k,v)}\\
\notag 
&\ind\ind= \ex{\int_{\R_+\times (0,1]\times\jmps\times W} H(s,\rho\rate^v_j(s,\alt{X},\alt{\vms}),j,v)\,P^{R,J}(ds,d\rho,dj,dv)} \\
\notag 
&\ind\ind= \ex{\int_{\R_+\times (0,1]\times \jmps\times W}H(s,\rho\rate^v_{j}(s,\alt{X},\alt{\vms}),j,v)\rate^v_j(s,\alt{X},\alt{\vms})\,ds\,d\rho\,\Sm{\jmps\times W}(dj,dv)} \\
&\ind\ind = \ex{\int_{\R_+^2\times \jmps\times W} \indic{r\leq \rate^v_j(s,\alt{X},\alt{\vms})}H(s,r,j,v)\,ds\,dr\,\Sm{\jmps\times W}(dj,dv)}, \nonumber
\end{align}
where the last equality uses 
the change of variables $r = \rho\rate^v_j(s,\alt{X},\alt{\vms})$. 
When combined with \eqref{dualitypf:poissdef}, the last two displays yield
\begin{align*}
&\ex{\int_{\R_+^2\times\jmps\times W} H(s,r,j,v)\alt{\poiss}_W(ds,dr,dj,dv)}\\
&=\ex{\sum_{v \in W}\int_{\R_+^2\times\jmps} \indic{r> \rate^v_j(s,\alt{X},\alt{\vms})}H(s,r,j,v)\wh{\poiss}(ds,dr,dj)}\\
&\ind\ind + \ex{\sum_{v \in W}\sum_{k \in \N} H(t^v_k,R^v_k\rate^v_j(t^v_k,\alt{X},\alt{\vms}),j^v_k,v)}\\
&=\ex{\int_{\R_+^2\times \jmps\times W} \indic{r> \rate^v_j(s,\alt{X},\alt{\vms})}H(s,r,j,v)\,ds\,dr\,\Sm{\jmps\times W}(dj,dv)} \\
&\ind\ind + \ex{\int_{\R_+^2\times \jmps\times W} \indic{r\leq \rate^v_j(s,\alt{X},\alt{\vms})}H(s,r,j,v)\,ds\,dr\,\Sm{\jmps\times W}(dj,dv)}\\
&=\ex{\int_{\R_+^2\times \jmps\times W} H(s,r,j,v)\,ds\,dr\,\Sm{\jmps\times W}(dj,dv)}.
\end{align*}
By \cite[Theorem 15.1.22]{Bre20}, since $H$ is an arbitrary left-continuous, nonnegative $\wh{\nfilt}$-mark predictable function, $\alt{\poiss}_W$ is a $\wh{\nfilt}$-Poisson process on $\R_+^2\times \jmps\times W$ with intensity measure $\leb^2\otimes \Sm{\jmps\times W}$. Because we fixed $W\subseteq U$ to be finite and arbitrary, it follows that $\alt{\poiss}$ is a $\wh{\nfilt}$-Poisson process on $\R^2_+\times\jmps\times U$ with intensity measure $\leb^2\otimes\Sm{\jmps\times U}$.

\skipLine
\textbf{Step 3: Show that $\alt{X}$ satisfies the SDE \eqref{SDE-dual} driven by $\alt{\poiss}$.}
\skipLine

For each $v \in V$, let $\alt{\poiss}_v(ds,dr,dj) \defeq \alt{\poiss}(ds,dr,dj,\{v\})$. To complete the proof of the proposition, note that by \eqref{dualitypf:poissdef}, for every $v \in U$ and $t \in \R_+$, 
\begin{align*}
\alt{X}_v(0) + \int_{(0,t]\times\R_+\times\jmps} j \indic{r\leq \rate^v_j(s,\alt{X},\alt{\vms})}\alt{\poiss}_v(ds,dr,dj) &= \alt{X}_v(0) + \sum_{\substack{k \in \N\\ t_k^v \in (0,t]}} j^v_k\indic{R^v_k\rate^v_{j_k^v}(t_k^v,\alt{X},\alt{\vms}) \leq \rate^v_{j_k^v}(t_k^v,\alt{X},\alt{\vms})}\\
&= \alt{X}_v(0) + \sum_{s \in \jmp{\alt{X}_v}{t}} \Delta \alt{X}_v(s)\\
&= \alt{X}_v(t),
\end{align*}
which proves \eqref{SDE-dual}. This completes the proof of the proposition.
\end{proof}

\subsection{Proof of Proposition \ref{rn:girs}}
\label{rnpf}

As alluded to earlier, the main difficulty in the proof of Proposition \ref{rn:girs} 
is that $X$ and $\wh{X}^W$ will typically have explosive duals,
and thus we cannot apply standard point process change of measure theorems directly.
Instead, the proof will exploit the duality results of the previous section along with a classical change of measure result from \cite{Bre20}, 
which is reproduced in Section \ref{sec-BreThm} below. 

\subsubsection{A Classical Change-of-Measure Result}
\label{sec-BreThm} 

For convenience we rephrase here the result of \cite[Theorem 15.2.7]{Bre20} in the specific case that the mark space is finite, also taking $L(0)$ therein to be $1$. This result relies on the notion of local characteristics of a marked point process 
which we now define specialized to the case (relevant to us) when the mark space of the point process is finite.
\begin{definition} 
\label{rnpf:localchar}
Given a finite space $\Pol$ equipped with the counting measure $\Sm{\Pol}$, fix a complete, filtered probability space $(\Omega,\nfiltm,\nfilt,\eta)$ and let $N_Z$ be a $\nfilt$-adapted marked point process on $\R_+\times \Pol$ with $\nfilt$-intensity $(t,z) \mapsto \lambda(t,z)$. If there exist functions $t\mapsto \lambda_g(t) \defeq \sum_{z \in \Pol}\lambda(t,z)$ and $\Psi: \R_+\times \Pol \to \R_+$ such that for any $t \in \R_+$,
\[\lambda(t,z) = \lambda_g(t)\Psi(t,z) \te{ for all }z \in \Pol,\]
and $\sum_{z \in \Pol}\Psi(t,z) = 1$. Then we say that $N_Z$ admits the $(\eta,\nfilt$)-local characteristics $(\lambda_g,\Psi)$.
\end{definition}

\begin{theorem}{\cite[Theorem 15.2.7]{Bre20}} 
\label{rnpf:BreGirs}
Fix a probability space $(\Omega, \nfilt, \eta^1)$ and let $N_Z$ be a simple and locally finite point process on $\R_+$ with marks in $\Pol$. Suppose that $N_Z$ is $\nfilt$-adapted and admits the $(\eta^1,\nfilt)$-local characteristics $(\lambda_g,\Psi)$. Let $\{\theta(t)\}_{t \geq 0}$ be a nonnegative
$\nfilt$-predictable process and let $\{h(t,z)\}_{t \geq 0,z \in \Pol}$ be a nonnegative, $\nfilt$-mark predictable random function. Suppose that for all $t \geq 0$,
\begin{equation}
\label{rnpf:localint}
\int_0^t \lambda_g(s)\theta(s)\,ds < \infty \te{ a.s.},
\end{equation}
and
\begin{equation}
\label{rnpf:probtrans}
\sum_{z \in \Pol} h(t,z)\Psi(t,z) = 1\te{ a.s..}
\end{equation}
If the events of $N_Z$ are given by the sequence $\{(t_n,z_n)\}$, then define
\begin{equation}
\label{rnpf:LBreDef}
L(t) \defeq \left(\prod_{t_n \in (0,t]} \theta(t_n)h(t_n,z_n)\right)\exp\left(-\sum_{z \in \Pol}\int_{(0,t]} \left(\theta(s)h(s,z) - 1\right)\lambda_g(s)\Psi(s,z)\,ds\right).
\end{equation}
Then the following properties hold: 
\begin{enumerate}[(a)]
\item $\{L(t)\}_{t \in \R_+}$ is a nonnegative $\nfilt$-local martingale under $\eta^1$. Moreover, if $\exnd{L(t)} =1$ for all $t \geq 0$, then it is a $\nfilt$-martingale under $\eta^1$. 
\item If $\exnd{L(T)} = 1$ for some $T>0$ and if the measure $\eta^2$ is defined by
\[\frac{d\eta^2_T}{d\eta^1_T} = L(T),\]
then under $\eta^2$, $N_Z$ admits the $(\eta^2,\nfilt)$-local characteristic $(\theta\lambda_g, h\Psi)$ on $[0,T]$.
\end{enumerate}
\end{theorem}

\subsubsection{Proof of Proposition \ref{rn:girs}:}
\label{rnpf:pfpf}

\begin{proof}[Proof of Proposition \ref{rn:girs}:]
For simplicity of notation, we fix the finite set $W\subset V$, let $\Xref\defeq\Xref^W$ and $\mref\defeq\mref^W$. Let $\fpspace$ be the solution space associated with $\Xref$ and $\vms$ (in the sense of Definitions \ref{mod:drive} and \ref{mod:WP}). For all $t \in \R_+, z = (j,v)\in \jmps\times V$, define
\begin{equation}
\irate(t,z) \defeq \rate^v_j(t,\vms,\Xref)\quad \te{ and } \quad \irateref(t,z) \defeq \rateref^{W,v}_j(t,\vms,\Xref) \defeq \begin{cases}
1 &\te{ if } v \in W,\\
\rate^v_j(t,\vms,\Xref) &\te{ if } v \notin W.
\end{cases}
\label{rnpf:rateshort}
\end{equation}
By Lemma \ref{proper:a.s.}, the jump characteristics of $\Xref_W$ are a.s. well defined, so the quantity $L\defeq L^W$ from \eqref{rn:girsLt} is also well defined. Recall from the statement of the proposition that $\filt\supseteq \nfilt \defeq \hfiltm^{\vms}\vee\hfilt^{\Xref}$. From \eqref{rn:girsLt} it follows that $L$ is $\nfilt$-adapted and a.s. c\`adl\`ag. Fix $n \in \N$ and finite $U \subseteq V$ such that $W \subseteq U$. Let $\wh{P}$ be the dual of $\Xref$ and $\wh{P}_U\defeq \wh{P}|_{\R_+\times\jmps\times U}$ be the dual of $\Xref_U$ in the sense of Definition \ref{duo:dualdef}. Note that the dual $\wh{P}_U$ is locally finite because the jump characteristics of $\Xref_U$ exist and that the jump characteristics $\{(t_k,z_k) := (t_k, (j_k,v_k))\}_{k \in \mathbb{N}}$ are precisely the events of $\wh{P}_U$. By Proposition \ref{duo:duality}(a), $\wh{P}_U$ has $\nfilt$-intensity $\Lambdaref_{U}(t,z) \defeq \irateref(t,z)$ for all $(t,z) \in \R_+\times\jmps\times U$. Then the point process $\wh{P}$ admits the $(\PP,\nfilt)$-local characteristics $(\lambdaref_U,\wh{\Psi}_U)$ where
\[\lambdaref_{U}(t) \defeq \sum_{z' \in \jmps\times U} \irateref(t, z') \quad \te{ and } \quad \Psiref_{U}(t,z) \defeq \frac{\irateref(t,z)}{\lambdaref_{U}(t)},\qquad \te{ for }t \in \R_+\te{ and } z \in \jmps\times U.\]
Also, define $\lambda_{U}$ and $\Psi_{U}$ analogously, but with $\irate(\cdot,\cdot)$ in place of $\irateref(\cdot,\cdot)$, and for $t > 0$ and $z \in \jmps\times U$, set
\begin{align*}
\theta_U(t) \defeq \frac{\lambda_U(t)}{\lambdaref_U(t)} = \frac{\sum_{z' \in \jmps\times U} \irate(t,z')}{|\jmps||W| + \sum_{\zeta' \in \jmps\times (U\setminus W)} \irate(t,\zeta')} \quad\te{ and }\quad h_U(t,\zeta) \defeq \frac{\frac{\irate(t,z)}{\lambda_U(t)}}{\frac{\irateref(t,z)}{\lambdaref_U(t)}}, \quad t > 0, z \in \jmps \times U. 
\end{align*}

\ind Then by \eqref{rnpf:rateshort}, the quantity $L(t)$ from \eqref{rn:girsLt} can be rewritten in a form suitable for the application of Theorem \ref{rnpf:BreGirs} as
\begin{equation}
L(t) =\left(\prod_{t_n \in (0,t]}\theta_U(t_n)h_U(t_n,z_n)\right)\exp\left(-\sum_{z \in \Pol}\int_{(0,t]}(\theta_U(s)h_U(s,z) - 1)\lambda_U(s)\Psi_U(s,z)\,ds\right).
\label{rnpf:Lgendef}
\end{equation}
Note that since $\theta_U(t)h_U(t,z) = 1$ for $z \in \jmps\times (U\setminus W)$, $L$ does not depend upon the choice of $U$ so long as $W\subseteq U$. We now verify that the conditions of Theorem \ref{rnpf:BreGirs} are satisfied for $N_z \defeq \wh{P}_U, \Psi \defeq \wh{\Psi}_U, \lambda_g \defeq \lambdaref_U, \theta \defeq \theta_U$ and $h \defeq h_U$. By the definition of $\theta_U$ and $\lambdaref_U$ above, \eqref{rnpf:rateshort} and \eqref{rn:localint}, we have
\[\int_0^t \theta_U(s)\lambdaref_U(s)\,ds = \int_0^t \sum_{z' \in \jmps\times U}\irate(s,z')\,ds = \sum_{(j,v) \in \jmps\times U}\int_0^t\rate^v_j(s,\Xref,\vms)\,ds < \infty \te{ a.s.,}\]
which verifies \eqref{rnpf:localint}. Also, for any $t \in \R_+$,
\[\sum_{z \in \jmps\times U} h_U(t,z)\,\Psiref_U(t,z) = \sum_{z \in \jmps\times U} \frac{\irate(t,z)}{\lambda_U(t)} = 1,\]
which verifies \eqref{rnpf:probtrans}. Moreover, \eqref{rnpf:LBreDef} follows from \eqref{rnpf:Lgendef}. Furthermore, we have already shown that on the complete probability space $(\Omega, \nfilt,\PP)$, $\wh{P}_U$ is a locally finite point process on $\R_+$ with marks in $\jmps\times U$ that admits the $(\PP,\nfilt)$-local characteristics $(\lambdaref_U,\Psiref_U)$. Moreover, $\theta_U$ is nonnegative and $\nfilt$-predictable while $h_U$ is nonnegative and $\nfilt$-mark predictable. Since the conditions of Theorem \ref{rnpf:BreGirs} are satisfied, it follows that $L$ is a $\nfilt$-local martingale. This proves the first assertion of the proposition.

\ind Next, suppose that $L$ is also a $\nfilt$-martingale. Then $\exnd{L(t)} = \exnd{L(0)} = 1$ for all $t \in \R_+$. Fix a probability measure $\alt{\PP}$ on the space $(\Omega,\alt{\nfilt})$, where $\alt{\nfilt}$ is the $\alt{\PP}$-completion of $\nfilt$. For $t \geq 0$, fix $\PP_{t} \defeq \PP|_{\sigma(\vms,\Xref[t])}$ and $\alt{\PP}_{t} \defeq \alt{\PP}|_{\sigma(\vms,\Xref[t])}$ and suppose that for each $t \geq 0$, $d\alt{\PP}_{t}/d\PP_{t} = L(t)$ a.s.. Define $\alt{\mu}$ to be the law of $\law(\vms,\Xref)$ under $\alt{\PP}$. Because $L(t)$ does not depend on the choice of $U\supseteq W$, neither does $\alt{\PP}_t$. Fix $t\in \R_+$ and in the spirit of \eqref{rn:girsLt}, define $\alt{L}_t = \alt{L}^W_t:(\mksp\times\cad_t)^V \to \R_+$ by
\[\alt{L}^W_{t}(\dvms,x) \defeq \left[\prod_{0 < \thi{k}\leq t}\rate^{\vi{k}}_{\ji{k}}(\thi{k},\dvms,x)\right]\exp\left(-\sum_{(j,v) \in \jmps\times W}\int_{(0,t]}\left(\locrate^v_j(s,\dvms,x) - 1\right)\,ds\right),\]
if the jump characteristics $\{(\thi{k},\vi{k},\ji{k})\}$ of $x_W$ exist, and $0$ otherwise. Then $\alt{L}_t(\vms,\Xref[t]) = L(t)$ $\PP_t$-a.s.. So for any $A \in \borel((\mksp\times\cad_t)^{V})$,
\[\alt{\mu}_t(A) = \alt{\PP}_t((\vms,\Xref) \in A) = \exmu{\PP_t}{L(t)\indic{(\vms,\Xref)\in A}} = \int_A \alt{L}_t(\dvms,x)\,\mref_t(d\dvms,dx),\]
which shows that $\frac{d\alt{\mu}_t}{d\mref_t}(\vms,\Xref[t]) = L(t)$ $\PP$-a.s.. In fact, because $\Xref(t) = \Xref(t-)$ $\PP_t$ (and hence $\alt{\PP}_t$)-a.s., we also have
\[\frac{d\alt{\mu}_{t-}}{d\mref_{t-}}(\vms,\Xref[t)) = \frac{d\alt{\mu}_{t}}{d\mref_{t}}(\vms,\Xref[t]) = L(t) = L(t-) \te{ a.s..}\]
\ind To complete the proof of Proposition \ref{rn:girs}, it only remains to show that 
$\alt{\mu} = \law(\vms,X)$ for some weak solution $X$ to \eqref{mod:infpart}-\eqref{ic}. It follows from Theorem \ref{rnpf:BreGirs} that $\wh{P}_U$ admits the $(\alt{\PP},\alt{\nfilt})$-local characteristics
\[(\alt{\lambda}_U,\alt{\Psi}_U) = (\theta_U\lambdaref_U,h_U\Psiref_U) = (\lambda_U,\Psi_U).\]
Thus, under $\alt{\PP}$, for $(t,z) \in \R_+\times \jmps\times U$, $\wh{P}_U$ has $\alt{\nfilt}$-intensity
\[\alt{\Lambda}_U(t,z) \defeq \lambda_U(t)\Psi_U(t,z) = \rate(t,z) = \rate^v_j(t,\vms,\Xref),\]
for $z =(j,v) \in \jmps\times U$ where the last equality above uses \eqref{rnpf:rateshort}. However, since $\wh{P}_U = \wh{P}|_{\R_+\times\jmps\times U}$, and the above display holds for all finite $U\supseteq W$, this implies that $\wh{P}$ has $\alt{\nfilt}$-intensity
\begin{equation}
\alt{\Lambda}(t,z) = \rate^v_j(t,\vms,\Xref),
\label{rnpf:finallambda}
\end{equation}
for $z =(j,v) \in \jmps\times V$ with respect to $\alt{\PP}$.

\ind Note that $G$ is a deterministic graph with a countable vertex set, and the jump rate function family $\rateset$ is assumed to satisfy conditions 2 and 3 of Assumption \ref{mod:regular}. Moreover, because $\Xref_v(0) = \x(\vms_v)$ $\PP$-a.s. for all $v \in V$, and because $\alt{\PP}_0 \ll \PP_0$, the same must hold $\alt{\PP}$-a.s.. Then by \eqref{rnpf:finallambda}, under $\alt{\PP}$, the intensity of the dual $\wh{P}$ of $\Xref$ satisfies \eqref{duo:altint}. Thus, Proposition \ref{duo:duality}(b) states that, by extending the probability space if necessary, we may assume without loss of generality that $(\Omega,\filt,\alt{\PP})$ supports a filtration $\ov{\nfilt}\supseteq \alt{\nfilt}$ satisfying the usual conditions and a collection of i.i.d. $\ov{\nfilt}$-driving Poisson processes $\alt{\poiss} \defeq \{\alt{\poiss}_v\}_{v \in V}$ such that $\Xref$ satisfies \eqref{mod:infpart}-\eqref{ic} driven by $\alt{\poiss}$ $\alt{\PP}$-a.s. for the initial data pair $(\vms,\vtx)$. Thus, with respect to $\alt{\PP}$, $\Xref$ is a weak solution to \eqref{mod:infpart}-\eqref{ic} for the initial data $\vms$ and $\alt{\mu} = \law(\vms,\Xref)$. This concludes the proof of the proposition.
\end{proof}

\appendix

\section{Proof of the Conditional Independence Lemma}
\label{ap-condind}

The goal of this section is to prove Lemma \ref{ref:CIfact} from Section \ref{MRF:ref}. The proof relies on 
several technical properties of conditional independence established in
Lemmas \ref{condind:condind} and \ref{condind:conds}. 
We start with a basic measure-theoretic result, 
which establishes the equivalence of conditional expectations with respect to any $\sigma$-algebra $\nfilt$ and its completion $\ov{\nfilt}$.
\begin{lemma}
\label{condind:excomp}
Let $\pspace$ be a complete probability space and let $\nfiltm\subset \filtm$ be a sub-$\sigma$-algebra with $\PP$-completion $\ov{\nfiltm}$. Then, for any bounded, $\filtm$-measurable random variable $Z$, $\ex{Z|\nfiltm} = \ex{Z|\ov{\nfiltm}}$ a.s..
\end{lemma}
\begin{proof}
By definition, $\exnd{Z|\ov{\nfiltm}}: \Omega \to \R$ is a bounded, $\ov{\nfiltm}$-measurable function. Therefore, by \cite[Lemma 1.27]{Kal21}, there exists a $\nfiltm$-measurable function $g: \Omega \to \R$ such that $g(\omega) = \exnd{Z|\ov{\nfiltm}}(\omega)$ $\PP$-a.s.. Thus for any $A\in \nfiltm$, $A \in \ov{\nfiltm}$ and hence,
\[\ex{\indic{A}g} = \ex{\indic{A}\ex{Z|\ov{\nfiltm}}} = \ex{\ex{\indic{A}Z|\ov{\nfiltm}}}= \ex{\indic{A}Z}.\]
This proves that $\exnd{Z|\nfiltm} = g = \exnd{Z|\ov{\nfiltm}}$ a.s..
\end{proof}

\ind We now show that Lemma \ref{condind:excomp} immediately implies that conditional independence relations of $\sigma$-algebras are insensitive to completions. 
\begin{lemma}
\label{condind:condind}
Let $\nfiltm_i$, $i=1,2,3$ be (possibly incomplete) $\sigma$-algebras defined on a common complete probability space $\pspace$ with respective $\PP$-completions $\ov{\nfiltm}_i$, $i=1,2,3$. Then $\nfiltm_1\indp \nfiltm_2|\nfiltm_3$ if and only if $\ov{\nfiltm}_1\indp \ov{\nfiltm}_2|\ov{\nfiltm}_3$.
\end{lemma}
\begin{proof}
Suppose that $\nfiltm_1\indp \nfiltm_2|\nfiltm_3$ and let $\ov{A} \in \ov{\nfiltm}_1$. By \cite[Lemma 1.27]{Kal21}, there must exist a $\nfiltm_1$-measurable function $f: \Omega \to \R$ such that $f = \indic{\ov{A}}$ a.s.. Then the set $A\defeq f^{-1}(\{1\}) \in \nfiltm_1$ and the symmetric difference of $A$ and $\ov{A}$ is null. By application of Lemma \ref{condind:excomp} in the second and fourth equalities below, we see that
\[\PP(\ov{A}|\ov{\nfiltm}_{\{2,3\}}) = \PP(A|\ov{\nfiltm}_{\{2,3\}}) = \PP(A|\nfiltm_{\{2,3\}}) = \PP(A|\nfiltm_3) =\PP(A|\ov{\nfiltm}_3) = \PP(\ov{A}|\ov{\nfiltm}_3) \te{ a.s.}.\]
Thus, $\ov{\nfiltm}_1\indp \ov{\nfiltm}_2|\ov{\nfiltm}_3$. 

\ind The proof of the converse is similar, but in fact simpler. Suppose $\ov{\nfiltm}_1\indp \ov{\nfiltm}_2|\ov{\nfiltm}_3$. Then for any $A \in \nfiltm_1\subseteq \ov{\nfiltm}_1$, Lemma \ref{condind:excomp} implies
\[\PP(A|\nfiltm_{\{2,3\}}) = \PP(A|\ov{\nfiltm}_{\{2,3\}}) = \PP(A|\ov{\nfiltm}_3) = \PP(A|\nfiltm_3).\]
Thus, $\nfiltm_1\indp \nfiltm_2|\nfiltm_3$, which completes the proof.
\end{proof}

\ind 
The next ingredient is a list of basic properties about conditional independence, whose proofs, for example, can be found in \cite{PutSch85}. Note that \cite{PutSch85} includes the additional assumption that all $\sigma$-algebras considered are complete, but we can remove that assumption by repeated application of Lemma \ref{condind:condind}.
\begin{lemma}
\label{condind:conds}
Let $\nfiltm_i$, $i \in \{1,\dots,6\}$, be six $\sigma$-algebras defined on some common measure space. Then the following statements hold: 
\begin{enumerate}[(a)]
\item if $\nfiltm_1\indp \nfiltm_2|\nfiltm_3$, then $\nfiltm_2\indp\nfiltm_1|\nfiltm_3$ \cite[Proposition 2.4(a),(b)]{PutSch85};
\item if $\nfiltm_4\indp (\nfiltm_5,\nfiltm_6)$, then $\nfiltm_4\indp \nfiltm_5|\nfiltm_6$ \cite[Proposition 2.5(b)]{PutSch85};
\item if $\nfiltm_1\indp \nfiltm_2|\nfiltm_3$ and $\nfiltm_4\indp \vee_{i=1}^3\nfiltm_i$, then $\nfiltm_1\vee\nfiltm_4\indp \nfiltm_2|\nfiltm_3$ \cite[Proposition 3.2(d)]{PutSch85};
\item if $\nfiltm_1\indp (\nfiltm_4 \vee \nfiltm_5)|\nfiltm_3$, then $\nfiltm_1\indp \nfiltm_4|\nfiltm_3\vee\nfiltm_5$ \cite[Proposition 3.2(a)]{PutSch85};
\item if $\nfiltm_1\indp \nfiltm_2|\nfiltm_3$ and $\nfiltm_4\subseteq \nfiltm_1$, then $\nfiltm_4\indp \nfiltm_2|\nfiltm_3$ \cite[Theorem 3.1]{PutSch85};
\item if $\nfiltm_1\indp \nfiltm_2|\nfiltm_3$ and $\nfiltm_1\indp\nfiltm_4|\nfiltm_2\vee\nfiltm_3$, then $\nfiltm_1\indp \nfiltm_2|\nfiltm_3\vee\nfiltm_4$ \cite[Proposition 3.2(b)]{PutSch85};
\item if $\nfiltm_1\indp \nfiltm_2|\nfiltm_3$, $\nfiltm_4\subseteq \nfiltm_3$ and for every nonnegative, $\nfiltm_1$-measurable random variable $Z$, $\exnd{Z|\nfiltm_3}$ is $\nfiltm_4$-measurable, then $\nfiltm_1\indp \nfiltm_2|\nfiltm_4$ \cite[Theorem 3.3]{PutSch85}.
\end{enumerate}
\end{lemma}

\ind 
We are now ready to prove Lemma \ref{ref:CIfact}.
\begin{proof}[Proof of Lemma \ref{ref:CIfact}]
For convenience, we will freely apply the symmetry of conditional independence outlined in Lemma \ref{condind:conds}(a) without reference. Due to Lemma \ref{condind:condind},
property 1 of the lemma and \eqref{CI-3} imply the following: 
\begin{enumerate}[1'.]
\item $\hfiltm^{Z^2_i}$, $i =1,2,3,$ are mutually independent and independent of $\vee_{i=1}^3\hfiltm^{Z^1_i}$;
\item $\hfiltm^{Z^1_1}\indp \hfiltm^{Z^1_2}|\hfiltm^{Z^1_3}$.
\end{enumerate}
By property 1', it follows that $\hfiltm^{Z^2_3}\indp \vee_{i=1}^3\hfiltm^{Z^1_i}$. Then Lemma \ref{condind:conds}(b) (with $\nfiltm_4 = \hfiltm^{Z^2_3}$, $\nfiltm_5 = \hfiltm^{Z^1_1}$ and $\nfiltm_6 = \vee_{i=2,3}\hfiltm^{Z^1_i}$) yields $\hfiltm^{Z^1_1}\indp \hfiltm^{Z^2_3}|\vee_{i=2,3}\hfiltm^{Z^1_i}$. Together with the fact that $\hfiltm^{Z^1_1}\indp \hfiltm^{Z^1_2}|\hfiltm^{Z^1_3}$, which follows from property 2' above, Lemma \ref{condind:conds}(f) (with $\nfiltm_i = \hfiltm^{Z^1_i}$, $i=1,2,3,$ and $\nfiltm_4 = \hfiltm^{Z^2_3}$) implies that 
\begin{equation}
\label{aCI:1}
\hfiltm^{Z^1_1}\indp \hfiltm^{Z^1_2}|\vee_{i=1}^2\hfiltm^{Z^i_3}.
\end{equation}
By property 1', $\hfiltm^{Z^2_3}\indp \hfiltm^{Z^1_1}\vee \hfiltm^{Z^1_3}$. Hence, Lemma \ref{condind:conds}(b) (with $\nfiltm_4 = \hfiltm^{Z^2_3}, \nfiltm_5 = \hfiltm^{Z^1_1}$ and $\nfiltm_6 = \hfiltm^{Z^1_3}$) implies $\hfiltm^{Z^1_1}\indp \hfiltm^{Z^2_3}|\hfiltm^{Z^1_3}$. Thus, for any nonnegative $\hfiltm^{Z^1_1}$-measurable random variable $Y$, the quantity $\exnd{Y|\hfiltm^{Z^{\{1,2\}}_3}} = \exnd{Y|\hfiltm^{Z^1_3}}$
is $\hfiltm^{Z^1_3}$-measurable and, since $\hfiltm^{Z^1_3} \subset \hfiltm^{Z^3_3}$ by property 2 of the lemma, it is also $\hfiltm^{Z^3_3}$-measurable. Together with \eqref{aCI:1} and Lemma \ref{condind:conds}(g), (with $\nfiltm_i = \hfiltm^{Z^1_i}$, $i=1,2,$ $\nfiltm_3 = \vee_{i=1}^2\hfiltm^{Z^i_3}$ and $\nfiltm_4 = \hfiltm^{Z^3_3}$) this implies that
\begin{equation}
\label{aCI:2}
\hfiltm^{Z^1_1}\indp \hfiltm^{Z^1_2}|\hfiltm^{Z^3_3}.
\end{equation}
Next, property 1' implies that $\hfiltm^{Z^2_1}\indp \hfiltm^{Z^2_3}\vee(\vee_{i=1}^3\hfiltm^{Z^1_i})$. Since by property 2,
$\hfiltm^{Z^2_3}\vee(\vee_{i=1}^3\hfiltm^{Z^1_i}) 
\supseteq \hfiltm^{Z^1_1}\vee\hfiltm^{Z^1_2}\vee\hfiltm^{Z^3_3}$, this implies $\hfiltm^{Z^2_1}\indp\hfiltm^{Z^1_1}\vee\hfiltm^{Z^1_2}\vee\hfiltm^{Z^3_3}$.
Combined with \eqref{aCI:2} and Lemma \ref{condind:conds}(c) (with $\nfiltm_i = \hfiltm^{Z^1_i},$ $i=1,2,$ $\nfiltm_3 = \hfiltm^{Z^3_3}$ and $\nfiltm_4 = \hfiltm^{Z^2_1}$), this implies that $\hfiltm^{Z^1_1}\vee\hfiltm^{Z^2_1}\indp \hfiltm^{Z^1_2}|\hfiltm^{Z^3_3}$. On the other hand, property 1' and the fact that $\hfiltm^{Z^3_3}\subseteq \vee_{i=1}^2\hfiltm^{Z^i_3}$ (due to property 2) implies $\hfiltm^{Z^2_2}\indp (\vee_{i=1}^2\hfiltm^{Z^1_i})\vee\hfiltm^{Z^2_1}\vee\hfiltm^{Z^3_3}$. Applying Lemma \ref{condind:conds}(c) again (this time with $\nfiltm_1 = \vee_{i=1}^2\hfiltm^{Z^i_1}, \nfiltm_2 = \hfiltm^{Z^1_2}, \nfiltm_3 = \hfiltm^{Z^3_3}$ and $\nfiltm_4 = \hfiltm^{Z^2_2}$) yields 
\begin{equation}
\label{aCI:3}
\hfiltm^{Z^1_1}\vee\hfiltm^{Z^2_1}\indp \hfiltm^{Z^1_2}\vee\hfiltm^{Z^2_2}|\hfiltm^{Z^3_3}.
\end{equation} 
Next, given \eqref{aCI:3} we apply Lemma \ref{condind:conds}(e) with $\nfiltm_1 = \hfiltm^{Z^1_1}\vee\hfiltm^{Z^2_1}$, $\nfiltm_4 = \hfiltm^{Z^3_1}$ and note that $\nfiltm_4\subseteq \nfiltm_1$ by property 2 of the lemma to obtain $\hfiltm^{Z^3_1}\indp \hfiltm^{Z^1_2}\vee\hfiltm^{Z^2_2}|\hfiltm^{Z^3_3}$. Finally,
apply Lemma \ref{condind:conds}(e) again, but now with $\nfiltm_1 = \hfiltm^{Z^1_2}\vee\hfiltm^{Z^2_2}$, $\nfiltm_4 = \hfiltm^{Z^3_2}$ and also use $\nfilt_4\subseteq \nfilt_1$ again to conclude that 
\begin{equation}
\label{aCI:4}
\hfiltm^{Z^3_1}\indp \hfiltm^{Z^3_2}|\hfiltm^{Z^3_3}.
\end{equation}
The result then follows by Lemma \ref{condind:condind}.
\end{proof}

\section{Supplementary Properties of SGMRFs}
\label{MRFalt}

Recall the definition of SGMRFs given in the Introduction. We now describe some properties of SGMRFs. We begin with Lemma \ref{res:altdef}, which gives an equivalent definition of the SGMRF property. Lemma \ref{res:Markov} then applies this definition to prove the assertion made in Section \ref{intro} that the $\mdeg$-SGMRF property extends the notion of tree-indexed Markov chains.

\ind In the following lemma, we use the terminology of $\mdeg$-separation. For a positive integer $\mdeg$, given three disjoint subsets $A,B,S\subset V$, 
$S$ is said to $\mdeg$-separate $A$ and $B$ if any path between a vertex in $A$ and a vertex in $B$ contains a consecutive sequence of $\mdeg$ distinct vertices in $S$. 
\begin{lemma}
\label{res:altdef}
Let $G = (V,E)$ be a locally finite graph, and let $\Pol$ be a Polish space. Then a $\Pol^V$-random element $Z$ forms an $\mdeg$-SGMRF if and only if relation \eqref{intro:MRFprop}, namely
$Z_A\indp Z_B|Z_S$, holds for all finite, disjoint $A,B,S\subseteq V$ such that $S$ $\mdeg$-separates $A$ and $B$.
\end{lemma}
\begin{proof}
Throughout the proof, fix the $\Pol^V$-random vector $Z$. First, assume that $Z_A\indp Z_B|Z_S$ for all finite, disjoint $A,B,S\subseteq V$ such that $S$ $\mdeg$-separates $A$ and $B$ in $G$. Then for any disjoint partition $A',B',S'\subset V$ of $V$ such that $S' \defeq \kgneigh{\mdeg}{A'}{G}$ is finite and any finite $A\subseteq A'$ and $B\subseteq B'$, all paths from $A$ to $B$ must contain $\mdeg$ consecutive elements of $\kgneigh{\mdeg}{A'}{G}$, and so $S'$ $\mdeg$-separates $A$ and $B$ in $G$. Thus, $Z_{A}\indp Z_B|Z_{S'}$, and because $A,B$ are arbitrary finite subsets of $A'$ and $B'$ respectively, it follows that $Z_{A'}\indp Z_{B'}|Z_{S'}$. Thus $Z$ forms an $\mdeg$-SGMRF.

\ind To show the reverse implication, now assume that $Z$ forms an $\mdeg$-SGMRF and let $A,B,S\subseteq V$ be finite, disjoint vertex sets such that $S$ $\mdeg$-separates $A$ and $B$. Then we claim there exists a partition $A',B',S$ of $V$ that satisfies the following two conditions: 
\begin{enumerate}
\item[C1.] $A\subseteq A'$, $B\subseteq B'$ and $S\supseteq \kgneigh{\mdeg}{A'}{G}$;
\item[C2.] $\displaystyle Z_{A'}\indp Z_{(S\cup B')\setminus \kgneigh{\mdeg}{A'}{G}}|Z_{\kgneigh{\mdeg}{A'}{G}}$.
\end{enumerate}
If the claim holds, then an application of Lemma \ref{condind:conds}(d) with ${\mathcal G}_1 = {\mathcal H}^{Z_{A'}}$, ${\mathcal G}_3 = {\mathcal H}^{Z_{\kgneigh{\mdeg}{A'}{G}}}$, 
${\mathcal G}_4 = {\mathcal H}^{Z_{B'}}$, and ${\mathcal G}_5 = {\mathcal H}^{Z_{S \setminus \kgneigh{\mdeg}{A'}{G}}}$, along with the observation that $B$ and $\kgneigh{\mdeg}{A'}{G}$ are disjoint by assumption, and two applications of Lemma \ref{condind:condind}, shows that 
$Z_{A'}\indp Z_{B'}|Z_S$. By C1, this directly implies that $Z_A\indp Z_B|Z_{S}$ as desired. 

\ind Thus, it suffices to prove the claim. To this end, let $A'$ be the set of elements in $V\setminus S$ that are not $\mdeg$-separated from $A$ in $G$. Let $B' = V\setminus (S\cup A')$. Then $B \subseteq B'$ because $S$ $\mdeg$-separates $A$ and $B$. Moreover, by construction, $S$ $\mdeg$-separates $A'$ and $B'$. Thus, $\kgneigh{\mdeg}{A'}{G}\subseteq S$, and
C1 follows. 
Moreover, $A', (S\cup B)\setminus \kgneigh{\mdeg}{A'}{G}$ and $\kgneigh{\mdeg}{A'}{G}$ partition $V$, and $\kgneigh{\mdeg}{A'}{G}\subseteq S$ is finite by assumption. Thus, C2 also holds because $Z$ forms an $\mdeg$-SGMRF. That concludes the proof of the claim and therefore the lemma.
\end{proof}

\ind Given a locally finite tree $G = (V,E)$, a $\Pol^V$-random vector $Z$ is said to form a tree-indexed Markov chain if it is an MRF such that for any finite, connected $U\subseteq V$, $Z_U$ forms an MRF with respect to $G[U]$ (see, e.g., \cite[Chapter 12]{Geo11} or \cite[Section 2]{Zac83}).
\begin{lemma}
\label{res:Markov}
When $G$ is a tree, the $\Pol^V$-random vector $Z$ forms an $\mdeg$-SGMRF if and only if for any connected, finite $U\subseteq V$, $Z_U$ forms an $\mdeg$-MRF with respect to the graph $G\subg{U}$.
\end{lemma}
\begin{proof}
By uniqueness of paths in trees, note that for any finite, disjoint sets $A,B,S\subseteq V$, $S$ $\mdeg$-separates $A$ and $B$ in $G$ if and only if it $\mdeg$-separates $A$ and $B$ in $G\subg{U}$ for every finite connected $U\subseteq V$ such that $A\cup B\cup S\subseteq U$. Then by Lemma \ref{res:altdef}, this implies that $Z$ forms an $\mdeg$-SGMRF with respect to $G$ if and only if $Z_U$ forms an $\mdeg$-MRF with respect to $G\subg{U}$ for all finite, connected $U\subseteq V$.
\end{proof}

\ind
Clearly, all global MRFs are SGMRFs, which in turn (since $G$ is locally finite) are also MRFs, and
all three concepts coincide on finite graphs.
Furthermore, the SGMRF property is strictly stronger than the MRF property (see \cite[Corollary 11.33]{Geo11} and Lemma \ref{res:Markov}) 
and strictly weaker than the global MRF property (see \cite[Section 2]{Gol80}). Also see \cite[Example 8.24]{Geo11} as well as part B of the bibliographical notes of \cite[Section 8.2]{Geo11} for further discussion of MRFs and global MRFs. 

\section{The Existence of Point Process Duals}
\label{Ap:dual}
We now justify the existence of duals of weak solutions to the SDE \eqref{mod:infpart}-\eqref{ic}.
\begin{lemma}
\label{duo:dualexist}
Suppose the family of jump rate functions $\rateset$ satisfies conditions 2 and 3 of Assumption \ref{mod:regular}. Then for any (possibly infinite) $U\subseteq V$ and any weak solution $X$ to \eqref{mod:infpart}-\eqref{ic} for some initial data pair $(\vms,\vtx)$, the dual $P_U$ of $X_U$ exists a.s..
\end{lemma}
\begin{proof}
By Lemma \ref{proper:a.s.}, for each $v \in U$, the (countable sequence of) jump characteristics $\{(t_k^v,j_k^v,v)\}$ of $X_v$ are a.s. well-defined. It follows that the random measure $P_U$ defined by $P_U(\{(t,j,v)\}) = 1$ if and only if there exists a $k$ such that $t_k^v = t$ and $j_k^v = j$ (equivalently, $\Delta X_v(t) = j$) is a.s. expressible as a sum of a countable number of delta masses and is therefore a well-defined random integer-valued measure. To complete the proof, it remains to prove that $P$ is a point process, or equivalently that $P \in \Nms(\R_+\times\jmps\times U)$ a.s.. Recall from Section \ref{nota} that any bounded set $A \subseteq \R_+\times \jmps\times U$ is a subset of $[a,b]\times \jmps\times W$ for some $0\leq a < b < \infty$ and finite $W \subseteq U$. Note that $P_U([a,b]\times \jmps\times W) = |\jmp{X_W}{[a,b]}| = \sum_{v \in W}|\jmp{X_v}{[a,b]}|$, which is finite since $X_v$ is a $\cad$-valued random element. Therefore, a.s. for any bounded $A \subset \R_+\times \jmps\times U$, $P_U(A) < \infty$ which proves that $P_U$ is a.s. an element of $\Nms(\R_+\times\jmps\times U)$ and is therefore a point process.
\end{proof}

\section{Verification of the General Well-Posedness Assumption}
\label{WP}

In this section we prove Lemma \ref{lem-ass}. 
Specifically, given a graph $G = (V,E)$ that belongs to one of the classes specified in Theorem \ref{res:MRF}, and 
a jump rate function family $\rateset \defeq\{\rate^v_j\}_{v \in V, j \in \jmps}$ that satisfies Assumptions \ref{mod:regular} and \ref{mod:WPbd}, we show that the SDE \eqref{rn:infpart} is well-posed. For SDEs with jump rates that satisfy the additional symmetry (or automorphism invariant) condition stated in \cite[condition 1 of Definition 3.1]{GanRam-Hydro22}, well-posedness would follow from \cite[Theorem 4.2 and Propositions 5.15, 5.17]{GanRam-Hydro22}. Here, we use a simple trick to show that the latter theorem in fact also applies to SDEs with heterogeneous rates such as the modified rate function family $\wh{\rateset}^W \defeq \{\rateref^{W,v}_j\}_{v \in V,j \in\jmps}$ in \eqref{rn:infpart}. The idea is to assume without loss of generality that the vertices of $G$ are labeled by (distinct)
integers in $\N$, that is, we treat $G$ as a \emph{marked graph} in which each vertex of $G$ is equipped with an integer mark that is equal to its label. This ensures that each vertex has a unique mark so that the automorphism group of this new marked graph is trivial, and thus 
the symmetry condition from \cite{GanRam-Hydro22} automatically holds, so the result therein can be directly applied. The details are given below.

\begin{proof}[Proof of Lemma \ref{lem-ass}]
Fix $G$, $(\vms,\vtx)$ and the original jump rate function family $\rateset \defeq\{\rate^v_j\}_{v \in V, j \in \jmps}$ satisfying the stated assumptions. 
We begin by considering the modified rate functions with $W = \emptyset$ so that $\rateref^{W,v}_j = \rate^v_j$ for all $v \in V$ and $j \in \jmps$. Assume without loss of generality that $V\subseteq \N$ and recall from \cite[Section 6.1]{GanRam-Hydro22} that for any Polish space $\Pol$, $\wh{\Gmc}_*\sp{\{1\},\Pol}$ is the space of graphs whose vertex sets are subsets of $\N$ and whose vertices are equipped with marks lying in the space $\Pol$. By \cite[Lemma B.5]{GanRam-Hydro22}, this space is Polish. Then there exists a measurable, injective map $\psi_{\X}: \mksp^V \to \wh{\Gmc}_*\sp{\{1\},\N\times\mksp}$ given by $\psi_{\X}(\dvms) \defeq (G,(\mathbf{k},\dvms))$ where for each $v \in V$, $\mathbf{k}_v = v$. Likewise there exists a measurable, injective map $\psi_{\cad}: (\mksp\times\cad)^V \to \wh{\Gmc}_*\sp{\{1\},\N\times\mksp\times\cad}$ given by $\psi_{\cad}(\dvms,x) = (G,(\mathbf{k},\dvms,x))$ where $\mathbf{k}$ is defined similarly. We now proceed in four steps.

\skipLine
\textbf{Step 1: Construct a regular family of local rate functions in the sense of \cite[Definition 3.1]{GanRam-Hydro22}.}
\skipLine

Let $H = (V_H,E_H,\root_H)\in \wh{\Gmc}_{1,*}$, where $\wh{\Gmc}_{1,*}$ is the space of unmarked, rooted graphs of radius one (i.e. all non-root vertices are adjacent to the root). For each $v \in V$, let $H_v = (G\subg{\cl{v}},v) \in \wh{\Gmc}_{1,*}$ and let $I(H_v,H)$ be the set of (rooted) isomorphisms from $H_v$ to $H$. Then for each $j \in \jmps$, define the measurable function $\wlocrate^{H}_j: \R_+\times \cad^{V_H}\times \N^{V_H}\times \mksp^{V_H}\to \R_+$ by
\begin{equation}
\label{WP:locdef}
\wlocrate^{H}_j(t,x,\mathbf{k},\dvms) \defeq \sum_{v \in V}\sum_{\varphi\in I(H_v,H)}\locrate^v_j(t,(\dvms_{\varphi(u)})_{u \in \cl{v}},(x_{\varphi(u)})_{u \in \cl{v}})\indic{k_w = \varphi^{-1}(w), w \in V_H}
\end{equation}
for $(t,x,\mathbf{k},\dvms) \in \R_+\times\cad^{V_H}\times \N^{V_H}\times \mksp^{V_H}$, where $\locrate^v_j$ is the local jump rate function from condition 1 of Assumption \ref{mod:regular}. Note that for any $H$ and $\mathbf{k}$, at most one term in the double sum above will be non-zero. 

\ind Using \eqref{WP:locdef}, it can be directly verified that $\ov{\rateset} \defeq \{\wlocrate^H_j\}_{H \in \wh{\Gmc}_{1,*},j\in \jmps}$ satisfies condition 1 of \cite[Definition 3.1]{GanRam-Hydro22}. Condition 2 of \cite[Definition 3.1]{GanRam-Hydro22} follows from condition 2 of Assumption \ref{mod:regular}, so $\ov{\rateset}$ is a family of regular local jump rate functions in the sense of \cite[Definition 3.1]{GanRam-Hydro22}. This concludes step 1.

\skipLine
\textbf{Step 2: Show that the pair $(\vms,X)$ satisfies \eqref{mod:infpart}-\eqref{ic} if and only if $\psi_{\cad}(\vms,X)$ satisfies \cite[(3.3)]{GanRam-Hydro22}.}
\skipLine

For any locally finite rooted graph $G' = (V',E',\root',(w',\dvms'))$ with vertex marks in $\N\times\mksp$, define the family of jump rate functions $\rateset^{G'} \defeq \{\rate^{G',v}_j\}_{v \in V',j \in\jmps}$ as in the standing assumption of \cite{GanRam-Hydro22} with respect to the regular family of local jump rate functions $\ov{\rateset}$. Then for any $(t,\dvms,x) \in \R_+\times (\mksp\times \cad)^V$, $v \in V$ and $j \in \jmps$, \eqref{WP:locdef} and condition 1 of Assumption \ref{mod:regular} imply
\begin{equation}
\rate^{\psi_{\X}(\dvms),v}_j(t,x) = \wlocrate^{H_v}(t,x_{\cl{v}},(w)_{w \in \cl{v}},\dvms_{\cl{v}}) = \rate^v_j(t,\dvms,x).
\label{WP:congrates}
\end{equation}
Therefore $\rateset$ can be extended to a family of jump rate functions that satisfy the standing assumption of \cite{GanRam-Hydro22}. Moreover, if $\poiss^G_v = \poiss_v$ for every $v \in V$, it follows that $X$ satisfies \eqref{mod:infpart}-\eqref{ic} for the initial data pair $(\vms,\vtx)$ a.s. if and only if $\psi_{\cad}(\vms,X)$ satisfies equation (3.3) of \cite{GanRam-Hydro22} for the initial data $(\psi_{\X}(\vms),(\vtx(\vms_v))_{v \in V})$.

\skipLine
\textbf{Step 3: Show that \eqref{mod:infpart}-\eqref{ic} is strongly well-posed for $G$, $(\vms,\vtx)$ and $\rateset$.}
\skipLine

Because $G$ is assumed to be either a graph of bounded maximal degree or an a.s. realization of a Galton-Watson tree whose offspring distribution has a finite first moment, $G$ is therefore finitely dissociable in the sense of \cite[Definition 5.11]{GanRam-Hydro22} by \cite[Propositions 5.15 and 5.17]{GanRam-Hydro22}. We have already shown that $\rateset^{\psi_{\X}(\vms)}$ satisfies the standing assumption of \cite{GanRam-Hydro22}. To verify that \cite[Assumption 1]{GanRam-Hydro22} also holds, let $G' = (V',E',\root,(w',\dvms'))$ be any rooted, locally finite graph with vertex marks in $\N\times\mksp$. For any graph $H$, let $|H|$ denote the number of vertices of the graph. Then the standing assumption of \cite{GanRam-Hydro22}, \eqref{WP:locdef}, condition 1 of Assumption \ref{mod:regular} and Assumption \ref{mod:WPbd} imply that for any $v'\in V'$, $j \in \jmps$, $x' \in \cad^{V'}$ and $t \in \R_+$,
\begin{align*}
\rate^{G',v'}_j(t,x') &= \wlocrate^{([G'[\gcl{v'}{G'}]],v')}_j(t,x'_{\cl{v'}},w'_{\cl{v'}},\dvms'_{\cl{v'}})\\
&\leq \sup_{v \in V: |\gcl{v}{G}| = |\gcl{v'}{G'}|}\sup_{(\dvms,x) \in (\mksp\times \cad)^V} \locrate^v_j(t,\dvms_{\cl{v}},x_{\cl{v}})\\
&\leq \sup_{v \in V: |\gcl{v}{G}| = |\gcl{v'}{G'}|}\sup_{(\dvms,x) \in (\mksp\times \cad)^V}\rate^v_j(t,x,\dvms)\\
& \leq C(|\gcl{v'}{G'},t).
\end{align*}
Thus \cite[Assumption 1]{GanRam-Hydro22} holds. Then by \cite[Theorem 4.2]{GanRam-Hydro22}, \eqref{mod:infpart}-\eqref{ic} is strongly well-posed in the sense of \cite[Definition 3.7]{GanRam-Hydro22} for the initial data $(\psi_{\X}(\dvms),(\vtx(\vms_v))_{v \in V})$. This means for any filtration-Poisson process pair $(\filt,\poiss^G)$ in the sense of \cite[Remark 3.5]{GanRam-Hydro22}, there there exists an a.s. unique weak solution to \cite[(3.3)]{GanRam-Hydro22}. By step 2, this implies that for the filtration $\filt$ and Poisson processes $\{\poiss_v\}_{v \in V} \defeq\{\poiss^G_v\}_{v \in V}$, there exists an a.s. unique solution to \eqref{mod:infpart}-\eqref{ic}. Thus, \eqref{mod:infpart}-\eqref{ic} is also strongly well-posed.

\skipLine
\textbf{Step 4: Generalize the result to $W$ non-empty.}
\skipLine

We have shown that the SDE \eqref{mod:infpart}-\eqref{ic} is strongly well-posed when the jump rate function family $\rateset$ satisfies Assumptions \ref{mod:regular} and \ref{mod:WPbd}. However, note that when $\rateset$ satisfies these assumptions, then $\wh{\rateset}^W \defeq \{\rateref^{W,v}_j\}_{v \in V,j\in\jmps}$ also satisfies Assumptions \ref{mod:regular} and \ref{mod:WPbd}. Therefore, \eqref{mod:infpart}-\eqref{ic} is strongly well-posed even when $\rateset$ is replaced by $\wh{\rateset}^W$.
\end{proof}

\bibliographystyle{plain}
\bibliography{reference}
\end{document}